\documentclass{amsart}

\usepackage{hyperref}
\usepackage{amssymb}

\newcommand{\Dav}{\mathsf{D}}
\newcommand{\s}{\mathsf{s}}
\newcommand{\e}{\mathsf{e}}

\newcommand{\N}{\mathbb{N}}
\newcommand{\Z}{\mathbb{Z}}
\newcommand{\F}{\mathbb{F}}
\newcommand{\vs}{\mathsf{v}}

\DeclareMathOperator{\ord}{ord}

\newtheorem{theorem}{Theorem}[section]
\newtheorem{lemma}[theorem]{Lemma}
\newtheorem{corollary}[theorem]{Corollary}
\newtheorem{proposition}[theorem]{Proposition}

\theoremstyle{definition}
\newtheorem{definition}[theorem]{Definition}

\numberwithin{equation}{section}

\begin{document}

\title[Multi-wise and constrained fully weighted Davenport constants]{Multi-wise and constrained fully weighted Davenport constants and interactions with coding theory}

\author{Luz E. Marchan \and  Oscar Ordaz \and Irene Santos \and Wolfgang A. Schmid}

\address{(L.E.M) Departamento de Matem\'aticas, Decanato de Ciencias y Tecnolog\'{i}as, Universidad Centroccidental Lisandro Alvarado, Barquisimeto, Venezuela}
\address{(O.O. \& I. S.) Escuela de Matem\'aticas y Laboratorio MoST, Centro ISYS, Facultad de Ciencias,
Universidad Central de Venezuela, Ap. 47567, Caracas 1041--A, Venezuela}
\address{(W.A.S.) Universit\'e Paris 13, Sorbonne Paris Cit\'e, LAGA, CNRS, UMR 7539, Universit\'e Paris 8, F-93430, Villetaneuse, France}

\email{luzelimarchan@gmail.com}
\email{oscarordaz55@gmail.com}
\email{iresantos@gmail.com}
\email{schmid@math.univ-paris13.fr}

\thanks{The research of O. Ordaz is supported by the Postgrado de la Facultad de Ciencias de la U.C.V., the CDCH project, and the Banco Central de Venezuela; the one of W.A. Schmid by the ANR project Caesar, project number ANR-12-BS01-0011.}



\begin{abstract}
We consider two families of weighted zero-sum constants for finite abelian groups. For a finite abelian group $( G , + )$, a set of weights $W \subset \mathbb{Z}$, and an integral parameter $m$, the $m$-wise Davenport constant with weights $W$ is the smallest integer $n$ such that each sequence over $G$ of length $n$ has at least $m$ disjoint zero-subsums with weights $W$.  And, for an integral parameter $d$, the $d$-constrained  Davenport constant with weights $W$ is the smallest $n$ such that each sequence over $G$ of length $n$ has a zero-subsum with weights $W$ of size at most $d$. First, we establish a link between these two types of constants and several basic and general results on them. Then, for elementary $p$-groups, establishing a link between our constants and the parameters of linear codes as well as the cardinality of cap sets in certain projective spaces, we obtain various explicit results on the values of these constants.
\end{abstract}

\maketitle

\section{Introduction}

For a finite abelian group $(G,+)$ the Davenport constant of $G$ is the smallest $n$ such that each sequence $g_1 \dots g_n$ over $G$ has a non-empty subsequence the sum of whose terms is $0$. This is a classical example of a zero-sum constant over a finite abelian group, and numerous related invariants have been studied in the literature. We refer to \cite{gaogersurvey} for a survey of the subject.

An example are the multi-wise  Davenport constants: the $m$-wise Davenport constant, for some integral parameter $m$, is defined like the Davenport constant yet instead of asking for \emph{one} non-empty subsequence with sum $0$ one asks for $m$ disjoint non-empty subsequences with sum $0$. 
These constants were first considered by Halter-Koch \cite{HK} due to their relevance in a  quantitative problem of non-unique factorization theory. Delorme, Quiroz and the second author \cite{DOQ} high-lighted the relevance of these constants when using the inductive method to determine the Davenport constant itself.

Another variant are constants defined in the same way as the Davenport constant, yet imposing a constraint on the length of the subsequence whose sum is $0$.
A classical example are generalizations of the well-known Erd\H{o}s--Ginzburg--Ziv theorem, where one seeks zero-sum subsequences of lengths equal to the exponent of the group or also equal to the order of the group.
Another common constraint is to impose an upper bound on the length, for example again the exponent of the group, yielding the $\eta$-invariant of the group. Here, we refer to these constants as constrained Davenport constants, more specifically the $d$-constrained  Davenport constant is the constant that arises when one asks for the existence of a non-empty zero-sum subsequence of length at most $d$. There are numerous contributions to this problem and we refer to \cite[Section 6]{gaogersurvey} for an overview. 

In addition to these classical zero-sum constants, in recent years there was considerable interest in weighted versions of these constants. There are several ways to introduce weights in such problems. One that received a lot of interest lately is due to Adhikari et al. (see \cite{adetal,adhi0,thanga-paper,YZ} for some  contributions, and \cite{ZY} for a more general notion of weights)  where for a given set of weights $W \subset \Z$ one asks for the smallest $n$ such that a sequence $g_1 \dots g_n$ over $G$ has a $W$-weighted subsum that equals $0$, that is there exists a subsequence $g_{i_1}\dots g_{i_k}$ and $w_j \in W$ such that $\sum_{j=1}^k w_j g_{i_j}=0$, yielding the $W$-weighted Davenport constant of $G$. And, analogously, one defines the $m$-wise $W$-weighted Davenport constant and the $d$-constrained $W$-weighted Davenport (see Definition \ref{def_main} for a more formal definition).  We refer to \cite{grynk_book} for an overview on weighted zero-sum problems, including a more general notion of weights, and to \cite{halterkoch} for an arithmetical application of a weighted Davenport constant. 

In the present paper we investigate multi-wise weighted Davenport constants and constrained weighted Davenport constants. We obtain some results for general sets of weights $W$ and general finite abelian groups, but our focus is on elementary $p$-groups and on the case that the set of weights is ``full,'' that is it contains all integers except multiples of the exponent of the groups, which is the largest set of weights for which the problem is non-trivial (see Section \ref{sec_prel} for details). This work builds on and generalizes earlier work on the classical versions of these constants  by Cohen and Z{\'e}mor \cite{CZ},  Freeze and the last author \cite{FS}, and Plagne and the last author \cite{PS} for elementary $2$-groups; in particular we establish that the problems can be linked to problems in coding theory, as was known in the case of elementary $2$-groups.

Moreover, for the case of elementary $3$-groups the problem coincides with the plus-minus weighted  problem, that is the problem for sets of weights $\{+1, -1\}$. Recently, constrained Davenport constants were investigated in this case by Godhino, Lemos, and Marques \cite{GLM}, and we improve several of their results.

The organization of the paper is as follows. After recalling some standard terminology, we recall in Section \ref{sec_general} the definitions of the key invariants for this paper, and prove some general results on them.
In particular, we show that for arbitrary sets of weights the multi-wise weighted Davenport constants are eventually arithmetic progressions. This generalizes a result of Freeze and the last author \cite{FS} for the classical case.

We then focus on the fully-weighted case. First, in Section \ref{sec_codesgeneral}, we explain the link to coding theory in a general way. The crucial difference to earlier works such as \cite{CZ, GT, PS} is that here we are not restricted to binary linear codes and elementary $2$-groups, but more generally can consider $p$-ary linear codes for some prime $p$ and elementary $p$-groups (we deviate from the more usual convention to talk about $q$-ary codes, since on the one hand we only consider primes not prime powers and on the other hand to stay in line with the common usage for groups). We use this link in two different ways. In Section \ref{sec_exactvalue} we use known results on the optimal minimal distance of codes of small dimension and length to obtain the exact values of or good bounds for the constrained Davenport constants of some elementary $p$-groups for small exponent and rank. In that section, among other results, we use the fact that linear codes of minimal distance four are closely linked to cap sets in projective spaces. In particular, we obtain a characterization of the $3$-constrained fully-weighted Davenport constants of elementary $p$-groups in terms of the maximal cardinality of cap sets in \emph{projective} spaces over the fields with $p$ elements. For context, we recall that a link between a certain classical zero-sum problem and the cardinality of caps in \emph{affine} spaces was known (see \cite{edeletal}). In the other direction, in Section \ref{sec_asymptotic}, we use asymptotic bounds on the parameters of codes to obtain asymptotic bounds on our constants. In that section we also obtain some lower bounds for our constants that can be interpreted as existence-proofs for a certain type of codes related to the notion of intersecting codes introduced by Cohen and Lempel \cite{CL}.
Moreover, in Section \ref{sec_allmultiwise} we determine all the fully-weighted multi-wise Davenport constants for elementary $p$-groups of rank at most $2$, and for $C_3^3$ (for the values of the multi-wise Davenport constant in the classical case for these groups see \cite{HK} and \cite{BSP}, respectively). The last example shows an interesting additional phenomenon and  illustrates the difficulty of obtaining more general results.

\section{Preliminaries}
\label{sec_prel}

We recall some standard terminology and notation. We denote by $\N$ the positive integers and we set $\N_0=\N \cup \{0\}$. We denote by $\log$ the natural logarithm and by $\log_b$ the logarithm to the base $b$. For abelian groups, we use additive notation and we denote the neutral element by $0$. For $n \in \N$ let $C_n$ denote a cyclic group of order $n$. For $p$ a prime number we denote by $\F_p$ the field with $p$ elements. For each finite abelian group there exist uniquely determined $1 < n_1 \mid \dots \mid n_r$ such that $G \cong C_{n_1} \oplus \dots  \oplus C_{n_r}$.
One calls $n_r$ the exponent of $G$, and it is denoted by $\exp(G)$; moreover $r$ is called the rank of $G$. We say that $G$ is a $p$-group if its exponent is a prime power and we say that $G$ is an elementary $p$-group if the exponent is a prime (except for the trivial group).

We recall that a (finite) abelian group is a $\Z$-module, and also a $\Z / \exp(G) \Z$ module. In particular, an elementary $p$-group is a vector space over the field with cardinality $\exp(G)$ (except for the trivial group).

A family $(e_1 , \dots , e_k)$ of non-zero elements of a finite abelian group is called independent if $\sum_{i=1}^k a_i e_i = 0$ with $a_i \in \Z$ implies that $a_ie_i = 0$ for each $i$. For an elementary $p$-group, a family of non-zero elements is independent in this sense if and only if it is linearly independent when considering the group as vector space in the way given above. We call an independent generating set of non-zero elements a basis; in the case of elementary $p$-groups this notion of basis coincides with the usual one for vector spaces.

A key notion of this paper are sequences. We recall some terminology and notation. A sequence over $G_0$, a subset of a finite abelian group $G$, is an element of $\mathcal{F}(G_0)$, the free abelian monoid over $G_0$. We use  multiplicative notation.  This means that a sequence $S$ can be written uniquely as $\prod_{g \in G_0} g^{v_g}$ with $v_g \in \N_0$, or uniquely except for ordering as $S= g_1 \dots g_s$ where $g_i \in G_0$ and repetition of elements can occur. We denote by $|S|=s$ the length of the sequence, by $\sigma(S) = \sum_{i=1}^s g_i$ the sum of $S$, and by $\vs_g(S)= v_g$ the multiplicity of $g$ in $S$. Formally, a subsequence of $S$ is a divisor $T$ of $S$ in $\mathcal{F}(G_0)$, that is $T=\prod_{i \in I}g_i$ for some $I \subset \{1, \dots, s\}$, which thus matches the notion of subsequence used in other contexts.

We call subsequences $T_1, \dots, T_t$ of $S$ disjoint if $ \prod_{j=1}^t T_j \mid S$. This can be expressed as saying that there are pairwise disjoint subsets $I_j \subset \{1, \dots , s \}$, for $j \in \{1, \dots, t\}$, such that   $T_j=\prod_{i \in I_j}g_i$.  However, it should be noted that disjoint subsequences of a sequence can have elements in common; this can happen in case they appear with multiplicity greater than one in the original sequence. 

A sequence is called squarefree if no element appears with multiplicity greater than $1$. These sequences could be identified with sets; however, it is sometimes advantageous to keep the notions separate.

A sequence over $G_0$ is called a zero-sum sequence if its sum is $0 \in G$, and a zero-sum sequence is called a minimal zero-sum sequence if it is non-empty and does not have a proper and non-empty subsequence that is a zero-sum sequence.  We denote the set of all minimal zero-sum sequences over $G_0$ by $\mathcal{A} (G_0)$.

For a subset $W \subset \Z$ a $W$-weighted sum of $S$ is an element of the form $\sum_{i=1}^s w_i g_i$ with $w_i \in W$. We denote by $\sigma_W(S)$ the set of all $W$-weighted sums of $S$.  This notion extends the notion of sum in the classical case; indeed, $\sigma_{ \{ 1 \} }(S) = \{ \sigma(S) \}$. 
A $W$-weighted subsum of $S$ is a $W$-weighted sum of a non-empty subsequence of $S$; we choose to exclude the empty sequence, since  this is more convenient in general (in the rare cases we admit the empty sequence we mention it locally). The length of a subsum is just the length of the respective subsequence.  A $W$-weighted zero-sum (or zero-subsum) is merely a $W$-weighted sum (or subsum) whose value is $0 \in G$. 

Of course, the $W$-weighted sums of a sequence $S$ depend only on the image of $W$ under the standard map to $\mathbb{Z}/ \exp(G) \mathbb{Z}$. We thus could restrict to considering sets of weights contained in $\{0, 1, \dots, \exp(G)-1\}$.
The problems we study become trivial if $W$ contains $0$ or more generally a multiple of $\exp(G)$. Thus, 
we call a set of weights $W$ trivial with respect to $\exp(G)$ if it contains a multiple of $\exp(G)$.

As mentioned in the introduction in later sections we focus on the problem for the ``full'' set of weights $\mathbb{Z}\setminus  \exp(G) \mathbb{Z}$ or equivalently $\{1, \dots, \exp(G)-1\}$ that is the largest set of weights that is not trivial; we reserve the letter $A$ for this set of weights. Since in these investigations at least the exponent of the groups under consideration will always be clear and essentially fixed, no confusion should arise from the fact that the set $A$ depends implicitly on the exponent of the group.  

Moreover, we call a set of weights $W$ multiplicatively closed modulo  $\exp(G)$ if the image of $W$ under the standard map from $\mathbb{Z}$ to $\mathbb{Z}/ \exp(G) \mathbb{Z}$ is multiplicatively closed.

\section{Main definitions and general results}
\label{sec_general}

We begin by stating in a more formal and in part more general way the definition of the weighted versions of the constrained and multi-wise Davenport constants. For a discussion of earlier appearances of the $m$-wise Davenport constant we refer to the introduction;  in \cite{congruence} a definition of $\s_{W, L} (G)$ without weights was given.  

\begin{definition}
\label{def_main}
Let $G$ be a finite abelian group. Let $W \subset \Z$ be a non-empty subset.
\begin{enumerate}
\item For a non-empty set $L \subset \N$, the $L$-constrained $W$-weighted Davenport constant of $G$, denoted $\s_{W, L} (G)$, is the smallest $n \in \N \cup \{ \infty \}$ such that each $S \in \mathcal{F}(G)$ with $|S|\ge n $  has a $W$-weighted zero-subsum of length in $L$. For the special case that $ L = \{1, \dots, d\}$, we denote the constant by $\s_{W, \le d} (G)$ and call it the $d$-constrained $W$-weighted Davenport constant.
\item  For $m \in \N$, the $m$-wise $W$-weighted Davenport constant of $G$, denoted $\Dav_{W, m} (G)$, is the smallest $n \in \N \cup \{ \infty \}$ such that each $S \in \mathcal{F}(G)$ with $|S|\ge n $  has at least $m$ disjoint  $W$-weighted non-empty zero-subsums.
\end{enumerate}
\end{definition}
It is easy to see that  $\Dav_{W, m} (G)$ is in fact always finite; for a characterization of the finiteness of $\s_{W, \le d} (G)$ see Lemma \ref{lem_swd_gen}.

Sometimes we use simplified versions of this notation for common special cases.   
When no set of weights is indicated, we mean the set of weights $W = \{1\}$. Moreover,  $\Dav_{W}(G)= \Dav_{W, 1}(G)$. For some further conventions see the discussion before Lemma \ref{lem_lb_general}.

We make some simple observations how these constants depend on the parameters; we continue to use the just introduce notations. Furthermore, let $H \subset G$ be a subgroup, let $W' \subset W$, $L' \subset L$, and let $m' \le m$ a positive integer.
Then 
\begin{itemize}
\item $\Dav_{W,m}(H) \le \Dav_{W,m}(G)$ and  $\s_{W, L} (H) \le \s_{W, L} (G)$.  
\item $\Dav_{W,m}(G) \le \Dav_{W',m}(G)$ and  $\s_{W, L} (G) \le \s_{W', L} (G)$.  
\item $\Dav_{W,m'}(G) \le \Dav_{W,m}(G)$ and  $\s_{W, L} (G) \le \s_{W, L'} (G)$.    
\end{itemize}

Some of the results below allow to refine the inequalities given above. The following results show that multi-wise weighted Davenport constants are eventually arithmetic progressions. In the classical case it is known that they are eventually arithmetic progressions with difference equal to the exponent of the group (see \cite[Lemma 5.1]{FS}); in the presence of weights the difference depends on the set of weights and the exponent of the group, in a form we make precise below. We mention that a somewhat similar phenomenon occurs in recent investigations  of a quantity related to the $m$-wise Davenport constant for  non-commutative finite groups \cite{hungary}.

\begin{theorem}
\label{thm_arithprogr}
Let $G$ be a finite abelian group. Let $W \subset \mathbb{Z}$ be a set of weights, and let $\overline{W}$ denote its image under the standard map from $\mathbb{Z}$ to $\mathbb{Z}/ \exp(G) \mathbb{Z}$. Then $(\Dav_{W,m}(G))_{m \in \mathbb{N}}$ is eventually an arithmetic progression with difference $\min \{|U| \colon U \in \mathcal{A}(\overline{W})\}$.
\end{theorem}

We point out that the proof of this result is not constructive, in the sense that it does not yield an actual upper bound on an $M$ such that  $(\Dav_{W,m}(G))_{m \ge M}$ is an arithmetic progression. This could likely be overcome in a similar way as in the classical case (see \cite[Proposition 6.2]{FS}). However, the argument should be somewhat lengthy and the bound very weak. Moreover, in Corollary \ref{cor_mWG} we see that in many cases of interest a not too bad bound can be obtained in another way. Thus, we do not pursue the question of obtaining a general explicit bound.

We split the proof of the result into several lemmas that are also useful in their own right. In all the results below let $G$ be a finite abelian group, let $W \subset \mathbb{Z}$ be a set of weights, and let $\overline{W}$ denote its image under the standard map to $\mathbb{Z}/ \exp(G) \mathbb{Z}$. Furthermore, let $\e_W(G)= \min \{|U| \colon U \in \mathcal{A}(\overline{W})\}$. This notation is chosen for this specific context only, to stress the role $\min \{|U| \colon U \in \mathcal{A}(\overline{W})\}$ plays in our context; it is a parameter to describe $W$-weighted zero-sum constants of $G$. 
The quantity  $\min \{|U| \colon U \in \mathcal{A}(H_0)\}$ for $H_0$ a subset of a finite abelian group $H$  also comes up in other contexts (see below).  

We make an additional definition that allows to phrase some results on the multi-wise Davenport constants in a concise way; the definition makes sense in view of the result above. 

\begin{definition}
\label{def_D0}
Let $G$, $W$, $\overline{W}$, and $\e_W(G)$ as above.
Then, let $\Dav_{W,0}(G)$ denote the integer such that $\Dav_{W,m}(G) = \Dav_{W,0}(G)+ m\e_W(G) $ for all sufficiently large $m$. Moreover, let $m_W(G)$ denote the minimal integer such that $\Dav_{W,m}(G) = \Dav_{W,0}(G)+ m\e_W(G) $ for each $m \ge m_W(G)$.
\end{definition}

We start by collecting some simple remarks on $\e_W(G)$. 
We have $\e_W(G)\le \exp(G)$. This follows from the fact that the Davenport constant of $\Z / \exp(G)\Z$ is $\exp(G)$.  Equality holds if and only if modulo $\exp(G)$ the set $W$ contains only a single element that generates $\Z / \exp(G)\Z$, that is, in the classical case. 
Moreover, in the plus-minus weighted case, that is $W= \{+1,-1\}$, as well as in the fully weighted case, that is $W= \{1, \dots, \exp(G)-1\}$, we have $\e_W(G) = 2$. 
Of course, there is a large variety of other possible values for $\e_W(G)$ in general, for example 
$\e_{ \{1, 2\}} (G)= \lceil \exp(G) / 2 \rceil$. Indeed, we recall that the problem of determining the minimal cross number, a notion similar to that of length, of minimal zero-sum sequences plays an important role in investigations on the elasticity (see \cite{CS_ZS}); for elementary $p$-groups the problems in fact coincide (up to a scaling constant) and we refer to \cite{CS_rend} for recent investigations related to this problem.  
      
Now, we investigate under which conditions on $d$ we have that $\s_{W, \le d}(G)$ is finite; the bound we give in case it is finite is rather crude, and  mainly given for definiteness. The actual value is investigated in latter sections in certain cases.

\begin{lemma}
\label{lem_swd_gen}
\
\begin{enumerate}
\item We have $\s_{W, \le d}(G) = \infty$ for $d < \e_W(G)$ and $\s_{W, \le d}(G)  \le (\e_W(G)-1) |G|+1$ for $d \ge \e_W(G)$.
\item We have $\s_{W, \le \Dav_W(G)} (G) = \Dav_W(G)$.
\end{enumerate}
\end{lemma}
\begin{proof}
1. Let $g \in G$ an element of order $\exp(G)$. We assert that for every $\ell \in \N$ the sequence $g^{\ell}$ does not have a nonempty $W$-weighted zero-subsum of length strictly less than $\e_W(G)$. 
We note that $\sum_{i=1}^r w_i g$ with $w_i \in W$ is $0$ if and only if $\sum_{i=1}^rw_i \equiv 0 \pmod{\exp(G)}$. In other words, $ \overline{w_1} \dots \overline{w_r} $ is a zero-sum sequence over $\Z/ \exp(G)\Z$. Thus, the minimal length of a nonempty $W$-weighted zero-sum only involving the group element $g$ is $\e_W(G)$, establishing our claim. Thus, $\s_{W, \le d}(G) = \infty$ for $d < \e_W(G)$.

Now, let $S$ be a sequence over $G$ of length at least $(\e_W(G)-1) |G|+1$. It contains a subsequence of the form $h^{\e_W(G)}$ for some $h \in G$.
Let $w_1, \dots, w_{\e_W(G)}$ such that their sum is $0$ modulo $\exp(G)$; the existence is guaranteed by the definition of $\e_W(G)$. Then, $\sum_{i=1}^{\e_W(G)}w_i h$ is a $W$-weighted zero-subsum of $h^{\exp(G)}$ and thus of $S$; note that the order of $h$ might not be equal to $\exp(G)$, yet it is always a divisor of $\exp(G)$ and this suffices. Thus, $\s_{W, \le d}(G)  \le (\e_W(G)-1) |G|+1$ for $d \ge \e_W(G)$.

2. Every sequence of length $\Dav_W(G)$ has a non-empty $W$-weighted zero-subsum, which of course has length at most $\Dav_W(G)$ and thus $\Dav_W(G) \le \s_{W, \le \Dav_W(G)} (G)$. The converse inequality is obvious.
\end{proof}

We point out that in the classical case, that is $W= \{1\}$, of course $\e_W(G)= \exp(G)$. The fact that this is the first value of $d$ for which the constant is finite in that case can be seen as one reason that  there is a particular focus on $\mathsf{s}_{W, \le \e_W(G)}(G)$, typically denoted $\eta_W(G)$ and also $\mathsf{s}_{W,\{\exp(G) \}}(G)$ typically denoted just $\mathsf{s}_W(G)$.    

\begin{lemma}
\label{lem_lb_general}
Let $d, \ell \in \mathbb{N}$ with $d \ge 2$ and  $\ell \le \s_{W,\le d-1}(G)$. Then $\Dav_{W, \lceil \ell/d  \rceil}(G) \ge \ell$.
In particular, for each $m \in \N$ we have $\Dav_{W, m}(G) \ge m \e_W(G)$.
\end{lemma}
\begin{proof}
Let $S$ be a sequence of length $\ell - 1$. Since $|S| < \s_{W,\le d-1}(G)$, it follows that each $W$-weighted zero-subsum of $S$ has length at least $d$. Since $\lceil \ell/d  \rceil d \ge \ell > |S|$, it follows that $S$ cannot have $\lceil \ell/d  \rceil$ disjoint $W$-weighted zero-subsums. Thus, $\Dav_{\lceil \ell/d  \rceil}(G) > |S|$, establishing the first part of the result. 

To see the second part, we recall from Lemma \ref{lem_swd_gen} that  $\s_{W,\le \e_W(G)-1}(G) = \infty$. Thus, for each $m \in \N$, we have $m \e_W(G) \le \s_{W,\le \e_W(G) - 1}(G)$ and thus by the first part with $\ell = m \e_W(G)$ and $d= \e_W(G)$, the claim follows. 
\end{proof}

We establish an upper-bound for $\Dav_{W,m+1}(G)$ involving $\Dav_{W,m}(G)$ and $\s_{W, \le d}(G)$.
The result is similar to a result in \cite[Proposition 3.1]{FS} for the classical case. However, the weighted version is minimally weaker; in the classical version we can take $\s_{W, \le d}(G)-1$ instead of $\s_{W, \le d}(G)$.

\begin{lemma}
\label{lem_ub_rec}
We have $\Dav_{W,m+1}(G) \le \min_{d \in \N}\max\{\Dav_{W,m}(G) +d, \s_{W,\le d}(G)\}$.
\end{lemma}
\begin{proof}
Let $S$ be a sequence over $G$ such that
\[
|S| \ge \min_{d \in \N}\max\{ \Dav_{W,m}(G) +d, \s_{W,\le d}(G) \}.
\]
Let $d_0\in \N$ such that the minimum is attained for $d_0$. Since $|S| \ge \max\{\Dav_{W,m}(G) +d_0, \s_{W,\le d_0}(G)\}$, it follows that $S$ has a $W$-weighted zero-subsum of length at most $d_0$; let $T$ denote the corresponding subsequence of $S$. Then, $|ST^{-1}| \ge \Dav_{W,m}(G)$.

By the very definition of $\Dav_{W,m}(G)$ the sequence $ST^{-1}$ has $m$ disjoint $W$-weighted zero-subsums. Thus, we have established the existence of $m+1$ disjoint $W$-weighted zero-subsums of $S$, showing that  $\Dav_{W,m+1}(G) \le \min_{d \in \N}\max\{\Dav_{W,m}(G) +d, \s_{W,\le d}(G)\}$.
\end{proof}

We now can prove Theorem \ref{thm_arithprogr}.

\begin{proof}[Proof of Theorem \ref{thm_arithprogr}]
For $m \ge (\e_W(G)-1) |G|+1=m_0$ we have, using Lemma \ref{lem_swd_gen} and Lemma \ref{lem_lb_general},
\[
\max\{\Dav_{W,m}(G) +\e_W(G), \s_{W,\le \e_W(G)}(G)\} = \Dav_{W,m}(G) +\e_W(G).
\]
Thus by Lemma \ref{lem_ub_rec} we have for $m \ge m_0$ that  $\Dav_{W,m+1}(G) \le \Dav_{W,m}(G) +\e_W(G))$. Consequently the sequence $(\Dav_{W,m}(G) - m \e_W(G) )_{m \ge m_0 }$ is non-increasing. And, by Lemma \ref{lem_lb_general} it is a sequence of non-negative integers. Therefore it is eventually constant, establishing the result.
\end{proof}

We collect some lemmas that relate the values of the constants we investigate for a group $G$ to those of some subgroup $H$ and the quotient group $G/H$. These results generalize results known in the classical case and sometimes also in more restricted cases with weights.

We start with two lower bounds, which among others can be useful to assert in certain cases that the value of 
$\Dav_{W, m}(G)$ or  $\s_{W,\le d}(G)$ are strictly greater than the respective values for a proper subgroup (in other words the extremal sequences with respect to these constants generate the group).  These bounds generalize results known in the classical case, see for example \cite[Section 6.1]{geroldingerhalterkochBOOK}.

\begin{lemma}
\label{lem_Dadd}
Let $m,m_1, m_2 \in \mathbb{N}$ such that  $m \ge m_1 + m_2 -1$.  Let $H$ be a subgroup of $G$. Then $\Dav_{W, m}(G) \ge \Dav_{W, m_1}( H ) + \Dav_{W, m_2}( G / H ) - 1$.
\end{lemma}
\begin{proof}
Let $S$ be a sequence over $H$ of length  $\Dav_{W, m_1}( H ) - 1$ that does not have $m_1$ disjoint $W$-weighted zero-subsums.
Let $T$ be a sequence over $G$ of length $\Dav_{W, m_2}( G / H ) - 1$ such that the sequence $\overline{T}$ over $G/H$ obtained from $T$ by applying the canonical epimorphism to each element in $T$ does not have $m_2$ disjoint $W$-weighted zero-subsums. 

We claim that $ST$ does not have $m$ disjoint $W$-weighted zero-subsums. We start by analyzing the types of $W$-weighted zero-subsums there are. Let  $R \mid ST$ such that $0 \in \sigma_{W}(R)$. There are $S' \mid S$ and $T' \mid T$ such that  $R=S'T'$. 
\begin{itemize}
\item If $T'$ is empty, then obviously $0 \in \sigma_W(S')$, and this yields a $W$-weighted zero-subsum of $S$
\item If $T'$ is non-empty, then, since $\sigma_W(T') \cap (- \sigma_W(S')) \neq \emptyset$ and since $\sigma_W(S') \subset H$, we get that $\sigma_W(T')$ contains an element from $H$. Thus, the image of $T'$ under the canonical epimorphism gives rise to a $W$-weighted subsum over $G/H$. 
\end{itemize}
Suppose we have $R_1 \dots R_v \mid ST$ with $R_i$ non-empty and $0 \in \sigma_W (R_i)$ for each $i$.  
We write $R_i = S_i T_i$ such that $S_1  \dots S_m \mid S$ and $T_1 \dots T_m \mid T$.  

For each $i$ we know that $ 0 \in \sigma_W(S_i)$ where $S_i$ is non-empty or $0 \in \sigma_W (\overline{T_i})$ where $T_i$ is non-empty. Let us denote the set of indices $i$ for which the former holds by $I_S$ and the complement of $I_S$ in $\{1, \dots, m\}$ by $I_T$. 

Since $\prod_{i\in I_S} S_i \mid S$ and $S$ has at most $m_1 - 1$ disjoint $W$-weighted zero-subsums wet get that $|I_S|\le m_1 - 1$. Since $\prod_{i\in I_T} T_i \mid T$ and $\overline{T}$ has at most $m_2 - 1$ disjoint $W$-weighted zero-subsums wet get that $v= |I_S| + |I_T| \le m_1 -1 +  m_2 - 1 < m$. 

Thus, $ST$  does not have $m$ disjoint   $W$-weighted zero-subsums  and  $\Dav_{W, m}(G)-1 \ge |ST| \ge ( \Dav_{W, m_1}( H ) - 1 )  + ( \Dav_{W, m_2}( G / H ) - 1)$, establishing our claim.  
\end{proof}

\begin{lemma}
\label{lem_sadd}
Let $d \in \mathbb{N}$. Let $H$ be a subgroup of $G$, then $\s_{W,\le d}(G)  \ge \s_{W,\le d}(H) + \s_{W,\le d}(G/H) - 1$.
\end{lemma}
\begin{proof}
Let $T$ be a sequence over $G$ whose image in $G/H$ does not have a $W$-weighted zero-subsum.
Then $T$ does not have a $W$-weighted subsum that is an element of $H$.
Thus, for each $S$ a sequence over $H$ that does not have a  $W$-weighted zero-subsum,
we have that $ST$ does not have a $W$-weighted zero-subsum.
It follows that  $\s_{W, \le d}(G)  > (\s_{W,\le d} ( H ) - 1) + (\s_{W, \le d} ( G / H ) - 1 )$ establishing the claim.
\end{proof}

We observe that this lemma yields some information on the structure of the sequences over $G$ of length $\s_{W,\le d}(G) - 1$ that do not yet have a $W$-weighted subsum of length at most $d$. 
Namely, it follows, denoting by $H$ the subgroup generated by the elements in such a sequence, that $\s_{W,\le d}(G/H) =1$. Under various circumstances this implies that $G/H$ is trivial, which means that  the elements in such a sequence generate the group $G$.  In particular, this is the case for $G$ an elementary $p$-group and $W$ a non-trivial set of weights modulo $p$.

The following two results establish `inductive' upper bounds on our constants in terms of the constants for a subgroup and the quotient group with respect to this subgroup. The two preceding lemmas could be thought of as `inductive' lower bounds; this terminology is not so common.  For an overview of the inductive method see \cite[Section 5.7]{geroldingerhalterkochBOOK}. Our results expand known results to this more general context, containing the classical ones as a special case; yet in fact Lemma \ref{lem_sind} is even more general than the existing results in the classical case. For these `inductive' results we need to impose some restriction on the sets of weights, namely that $W$ is multiplicatively closed modulo $\exp(G)$.

Before phrasing our actual results we explain the relevance of this condition and more generally the inductive method. This method consists of splitting a zero-sum problem for $G$ into a problem for a subgroup $H$ and a problem for the quotient group $G/H$. Let $\pi: G \to G / H$ denote the natural epimorphism. Let $S$ be a sequence over $G$. To find a zero-sum subsequence of $S$ one can proceed in the following way. 
One considers the sequence $\pi(S)$ over $G/H$, obtained by applying $\pi$ to each element in $S$, for simplicity of notation we denote $\pi(S)$ by $\overline{S}$ and do alike for subsequences.  
Suppose we can assert the existence of disjoint zero-sum subsequences  $\overline{S_1} \dots \overline{S_k} \mid \overline{S}$. The condition that $\overline{S_i}$ is a zero-sum sequence in $G/H$ means that the sum of $S_i$ is an element of $H$. Thus, denoting $s_i = \sigma (S_i)$, we have a sequence $s_1 \dots s_k$ over $H$. Now, if this sequence has a zero-sum subsequence $\prod_{i \in I}s_i$, then it follows that $\prod_{i \in I}S_i$ is a zero-sum subsequence of $S$.  

We point out that we have used that we can take the sum of a sequence in an iterated way, namely we have $\sigma( \sigma( S_1 ) \dots \sigma( S_k ) ) =  \sigma(S_1 \dots  S_k)$. It is at this point that trying to directly carry over the results from the classical case to the problem with weights would fail.  

We now phrase a generalization of this for the weighted case. 
Let $s_i \in \sigma_W( S_i )$ for each $i$, then for $W$ multiplicatively closed modulo $\exp(G)$, we have  
\begin{equation}
\label{eq_sigincl}
\sigma_W( s_1 \dots s_k ) \subset \sigma_W( S_1 \dots  S_k ).
\end{equation}
To see this just recall that an element of $\sigma_W( s_1 \dots s_k )$ is of the form $\sum_{i=1}^k w_i s_i$ with $w_i \in W$ while each $s_i$ is of the form $\sum_{j \in J_i} v_j^i g_j$ with $v_j^i \in W$ where 
$S_i = \prod_{j \in J_i} g_j $; and, since modulo $\exp(G)$ we have that $w_i v_j^i $ is again an element of $W$, we have that $\sum_{j \in \cup_{i=1}^k J_i } (w_i v_j^i) g_j$ is an element of $\sigma_W( S_1 \dots  S_k )$.  

We now formulate our inductive bounds for the two types of constants that we study. 
For the classical Davenport constant this result appears in \cite[Proposition 2.6]{DOQ}

\begin{lemma}
\label{lem_Dind}
Let $m \ge 1$ and let  $H$ be a subgroup of $G$. Suppose $W$ is multiplicatively closed modulo $\exp(G)$. Then
\[
\Dav_{W, m} ( G ) \le \Dav_{W, \Dav_{W, m}(H)}( G / H).
\]
\end{lemma}
\begin{proof}
For notational simplicity let $ D = \Dav_{W, \Dav_{W, m} (H)}(G / H) $.
Let $S$ be a sequence over $G$ of length $D$.  We need to show that $S$ has $m$ disjoint $W$-weighted zero-subsums. By the definition of $D$ and considering $\overline{S}$ the sequence obtained from $S$ by applying the canonical epimorphism from $G$ to $G/H$ to each element, it follows that $\overline{S}$ has $\Dav_{W, m}(H)$ disjoint $W$-weighted zero-subsums (in $G/H$). Thus, $S$ has $\Dav_{W, m}(H)$ disjoint subsequences $S_i$ such that $\sigma_W(S_i)$ contains an element from $H$; denote this  element by $s_i$.  

The sequence $T=s_1 \dots s_{\Dav_{W, m}(H)}$ is thus a sequence over $H$. By the definition of $\Dav_{W, m}(H)$, there exists $T_1 \dots T_m \mid T$ with $T_j$ non-empty and $0 \in \sigma_W(T_j)$ for each $j$. 
Let $I_j \subset \{1, \dots, \Dav_{W, m}(H)\}$, for $j\in \{1, \dots, m\}$, be disjoint subsets such that $T_j = \prod_{i \in I_j}s_i$. 
 
Since $W$ is multiplicatively closed, we have by \eqref{eq_sigincl} that 
\[
\sigma_{W} (\prod_{i \in I_j} s_i) \subset \sigma_{W} (\prod_{i \in I_j} S_i).
\] 
Consequently, $R_j= \prod_{i \in I_j} S_i$ for $j \in \{1, \dots, m\}$ are non-empty disjoint subsequences of $S$ such that $0  \in \sigma_{W}(R_j)$. Thus, we established the existence of $m$ pairwise disjoint $W$-weighted zero-subsums of $S$.
\end{proof}

\begin{lemma}
\label{lem_sind}
Let $L_1, L_2 \subset \mathbb{N}$ be non-empty subsets, and let  $H$ be a subgroup of $G$. Suppose $W$ is multiplicatively closed modulo $\exp(G)$. Then
\[
\s_{W,  L} ( G ) \le    (\max L_2) ( \s_{W, L_1}(H)  - 1 )   +  \s_{ W , L_2 }( G / H )
\]
where $L = \cup_{l \in L_1} l L_2$ and $l L_2$ denotes the $l$-fold sumset of $L_2$.
\end{lemma}
\begin{proof}
Let $S$ be a sequence over $G$ of length  $(\max L_2) ( \s_{W, L_1}(H)  - 1 )   +  \s_{ W , L_2 }( G / H )=t$. 
Let $\overline{S}$ be the sequence obtained from $S$ by applying the canonical epimorphism from $G$ to $G/H$ to each element. Since $t \ge \s_{W,  L_2}( G / H) $  there exists a subsequence $T_1 \mid S$  with $|T_1| \in L_2$, and in particular $ |T_1| \le \max L_2$, such that $ 0 \in \sigma_{W}(\overline{T_1})$ that is there exists some element $s_1 \in H$ such that  $s_1 \in \sigma_{W}(T_1)$.
We now consider $S_1 = T_1^{-1}S$.   If $|S_1| \ge  \s_{W, L_2}( G / H)$ we get $T_2 \mid S_1$  with the analogous properties. Continuing in this way we get $\s_{W, L_1 }(H) = c$ disjoint subsequences $T_1 \dots T_{c} \mid S$ with $|T_i| \in L_2$ and such that there exists some $s_i \in  \sigma_{W}(T_i) \cap  H$. By definition of $c$ the sequence $s_1 \dots s_{c}$  has a $W$-weighted subsum of length in $L_1$, say $0 \in \sigma_W( \prod_{j \in J}s_j)$ with $J \subset \{1, \dots,  c\}$ and $|J| \in L_1$.
Now, by \eqref{eq_sigincl}, since $W$ is multiplicatively closed modulo $\exp(G)$, we have
$\sigma_W( \prod_{j \in J}s_j) \subset  \sigma_{W} ( \prod_{j \in J} T_j) $ and 
thus $\prod_{j \in J}T_j$ has $0$ as a $W$-weighted sum. 
Since $|T_j| \in L_2$ for each $j$, we have that the length of $\prod_{j \in J}T_j$ is  in $|J| L_2 $, which is a subset of $\cup_{l \in L_1} l L_2$. Thus, we have found a $W$-weighted subsum of $S$ whose length is in $\cup_{l \in L_1} l L_2$, establishing the result. 
\end{proof}

We use the result above with special choices of $L_1$ and $L_2$, to make explicit some consequences of particular relevance to our investigations. The second part of this corollary generalizes \cite[Proposition 6]{GLM}, an inductive result in the weighted case. For the classical case one can find various such results in the literature, especially for $\mathsf{s}(G)$ and $\eta(G)$ they are well-known (see for example \cite[Section 6]{gaogersurvey}). For example, they allow in combination with results for $p$-groups, to determine the exact value of  $\mathsf{s}(G)$ and $\eta(G)$ for groups of rank at most $2$ (see \cite[Theorem 5.8.3]{geroldingerhalterkochBOOK}). For some recent results for groups of higher rank we refer to \cite{fanetal1, fanetal2, wzhuang}.   
 
\begin{corollary}
\label{cor_sind}
Let $d_1, d_2 \ge 1$ and let  $H$ be a subgroup of $G$. Suppose $W$ is multiplicatively closed modulo $\exp(G)$. Then
\begin{enumerate}
\item $\s_{ W,  \le d_1d_2 } ( G )    \le d_2 (\s_{W, \le d_1 }(H) -1  ) + \s_{W,  \le d_2}  ( G / H )$.
\item $\s_{ W,  \{ d_1d_2 \}} ( G )  \le d_2 (\s_{W, \{d_1\} }(H) -1) + \s_{W,  \{d_2\} } ( G / H )$.
\end{enumerate}
\end{corollary}
\begin{proof}
For 1., we apply Lemma \ref{lem_sind} with $L_1 =\{1, \dots, d_1\}$ and $L_2 = \{ 1, \dots, d_2 \}$. 
Since in this case  $\cup_{l \in L_1} l L_2 \subset \{1, \dots, d_1d_2 \}$ the result follows.

In the same way, with  $L_1 =\{ d_1\}$ and $L_2 =\{ d_2\}$ and thus  $L=\cup_{l \in L_1} l L_2 = \{d_1d_2 \}$, we get 2.
\end{proof}

We end this section with a recursive lower bound for $\Dav_{W, m}(G)$. At first this bound could seem weak, but in view of Lemma \ref{lem_ub_rec} we note that it is in some sense optimal in case  $\e_W(G)=2$, which covers various cases of interest (see the discussion after Definition \ref{def_D0}). 
Indeed, this bound allows us to obtain a more explicit version of Theorem \ref{thm_arithprogr}.  
We also note that $\Dav_{W,m+1}(G) \ge \Dav_{W,m}(G) +1$ always holds; adding $0$ to a sequence increases the maximal number of disjoint $W$-weighted zero-subsums by exactly one.

\begin{lemma}
\label{lem_lb_+2}
Let $|G|>1$ and let $m \in \N$. Suppose that $W$ is multiplicatively closed modulo $\exp(G)$ and non-trivial modulo $\exp(G)$. Then $\Dav_{W,m+1}(G) \ge \Dav_{W,m}(G) +2$.
\end{lemma}
\begin{proof}
Let $S$ be a sequence of length $\Dav_{W,m}(G) - 1$ that does not have $m$ disjoint $W$-weighted zero-subsums. Let $g \in G$ with $\ord g = \exp(G)$. We consider $S(-g)g$.
The result follows if we can show that $S(-g)g$ does not have $m+1$ disjoint $W$-weighted zero-subsums. Assume to the contrary there are $T_1 \dots T_{m+1} \mid S$ such that $0 \in \sigma_W(T_i)$ for each $i$. We may assume that $g \mid T_1$ and $(-g) \mid T_2$; if only one or none of the $T_i$'s would contain elements from $(-g)g$ then the remaining sequences would be already subsequences of $S$, contradicting the assumption that $S$ does not contain $m$ disjoint $W$-weighted subsums. Since $W$ is non-trivial modulo $\exp(G)$ and $\ord g = \exp(G)$, it is clear that $g^{-1}T_1$ and $(-g)^{-1}T_2$ are non-empty. 

Thus, we have $-wg \in \sigma_W(g^{-1}T_1)$ and $-w'(-g) \in \sigma_W((-g)^{-1}T_2)$.
Consequently,   $w'(-wg) \in  w' \cdot \sigma_W(g^{-1}T_1) = \sigma_W(g^{-1}T_1)$, where for the last equality we used that $W$ is multiplicatively closed modulo $\exp(G)$, and likewise $w(-w'(-g)) \in \sigma_W((-g)^{-1}T_2)$. 

Thus  $-ww'g  \in  \sigma_W(g^{-1}T_1) $ and  $ww'g \in \sigma_W((-g)^{-1}T_2)$, and therefore $0 \in \sigma_{W}(((-g)g)^{-1}T_1 T_2)$. Yet, $(((-g)g)^{-1}T_1 T_2) T_3 \dots T_{m+1} \mid S$, and we have $m$ disjoint $W$-weighted zero-subsums of $S$, a contradiction.
\end{proof}

For the  case $\e_W(G) = 2$ and $W$ multiplicatively closed modulo $\exp(G)$,   
we now obtain some explicit upper bound for $m_W(G)$, the index at which $\Dav_{W,m}(G)$ starts to be an arithmetic progression. 
\begin{corollary}
\label{cor_mWG}
Suppose that $W$ is multiplicatively closed modulo $\exp(G)$ and that $\e_W(G)= 2$.
Then  $\Dav_{W, m+1} (G) = \Dav_{W, m} (G) + 2$ for each $m \ge |G|$. In particular, $m_W(G) \le |G|$.  
\end{corollary}
\begin{proof}
This is a direct consequence of Lemma \ref{lem_ub_rec} and Lemma \ref{lem_lb_+2}. 
\end{proof}
We remark that the bound on $m_W(G)$ could be somewhat improved even with the methods at hand; yet we merely meant to give some explicit bound here.

\section{Coding theory and weighted sequences}
\label{sec_codesgeneral}

In this section, we develop the link between fully-weighted zero-sum problems and problems on linear codes that was mentioned already in the introduction. We recall that for elementary $2$-groups, this  link was already known;
note that in fact for these groups the fully-weighted problem coincides with the classical one, and the connection came up in that context (see \cite{CZ, PS}). Furthermore, we recall that for elementary $3$-groups the fully-weighted problem coincides with the plus-minus weighted problem, which is of particular interest. 

Before we discuss the link to coding theory we recall some basic facts related to fully-weighted zero-sum problems over elementary $p$-groups. Let $p$ be a prime number. Let $G$ be a group with exponent $p$, and let $A =\{1, \dots , p-1\}$ denote the full set of weights. Recall that $G$ can be considered in a natural way as a vector space over the field with $p$ elements. We also identify the elements of $A$  with the non-zero elements of $\F_p$ in the natural way.  Furthermore, we recall that a sequence $S = g_1 \dots g_n$ over $G$ has no $A$-weighted zero-subsum if and only if $(g_1, \dots, g_n)$ is linearly independent,
in particular
\begin{equation}
\label{eq_davA}
\Dav_A (C_p^r) = r+1  \text{ and } \s_{A, \le r+1} (C_p^r) = r + 1.
\end{equation}

Now, we briefly recall some notions from coding theory in a way suitable for our application. As said above, $C_p^n$ is in a natural way a vector space over $\F_p$ the field with $p$ elements. We implicitly fix some basis of  $C_p^n$, and thus can write its elements simply as $n$-tuples of elements of $\F_p$:
\[
C_{p}^{n}=\{(x_1,x_2,\dots,x_n) \colon x_i \in \F_p,\,\,1\leq i \leq n\}.
\]

A $p$-ary linear code  of length $n$ and dimension $k$ is a subspace $\mathcal{C} \subset C_{p}^{n}$ of dimension $k$.
Briefly, we say that  $\mathcal{C}$ is an $[n,k]_{p}$-code. The elements of the code $\mathcal{C}$ are called codewords. The support of an  element $x= (x_1, x_2, \dots, x_n)$ of $C_{p}^{n}$, is the subset of $\{1, \dots, n\}$ corresponding to the indices of non-zero coordinates $x_i$. (This is the usage of the word `support' common in coding theory; in the context of zero-sum sequences `support' typically has a different meaning.) The weight of an element $x\in C_p^n$,  denoted by $d(x)$, is the cardinality of the support of $x$. The minimal distance of a code $\mathcal{C}$, denoted $d(\mathcal{C})$, is equal to the minimum $d(x)$ with $x$ a non-zero codeword of $\mathcal{C}$; that is, the minimal distance of the code $\mathcal{C}$, is equal to the minimum cardinality of the support of a non-zero element of $\mathcal{C}$. If $\mathcal{C}$ has minimal-distance $d$, then we say that $\mathcal{C}$ is a $[n,k,d]_{p}$-code. Since in the present paper we only consider linear codes, we choose the above quick, but not very intuitive way, to introduce the minimal distance.

A parity check matrix of an $[n,k]_{p}$-code  $\mathcal{C}$ is a matrix $H$ of dimension $(n-k)\times n$ (with full rank) over $\F_p$ such that $c\in \mathcal{C}$ if and only if $Hc=0$ (where we consider $c$ as a column vector). For $H=[g_1|\dots|g_n]$, we can interpret the columns $g_i$ as elements of $C_{p}^{n-k}$, where again some basis is fixed, and in this way one can assign to an $[n,k]_{p}$-code a sequence $S=g_1\dots g_n$ of length $n$ over $C_{p}^{n-k}$. Now, $c = (c_1, \dots , c_n) \in \mathcal{C}$ means that $\sum_{i=1}^n c_i g_i = 0$. If we let $I$ denote the support of $c$, then we have  $\sum_{i\in I} c_i g_i = 0$ for \emph{non-zero} $c_i$. 

In other words $\sum_{i\in I} c_i g_i = 0$ is an $A$-weighted zero-subsum of $S$ (possibly the empty one); its length is exactly the cardinality of the support of $c$, that is $d(c)$. Conversely, if $ \sum_{i\in J} c_i' g_i = 0 $ is a (possibly empty) $A$-weighted zero-subsum of $S$, then $ c' = ( c_1' , \dots , c_n' ) $, where we set $c_i'= 0$ for $i \notin J$, is an element of $\mathcal{C}$ with weight $|J|$, the length of the subsum. Of course, $0 \in \mathcal{C}$ corresponds to the empty $A$-weighted zero-subsum. Thus, we have a direct correspondence between non-zero codewords and $A$-weighted zero-subsums.

We summarize these results in the lemmas below (recall $A = \{1,\dots , p - 1 \}$).

\begin{lemma}
\label{lem_codebasic}
The minimal distance of a $p$-ary linear code $\mathcal{C}$ is equal to  the minimal length of
a $A$-weighted zero-subsum of columns of a parity check matrix of $\mathcal{C}$.
\end{lemma}
\begin{proof}
This is immediate by the discussion just above.
\end{proof}

\begin{lemma}
\label{lem_codematrix}
Let $S = g_1 \dots g_n$ be a sequence over $C_{p}^{r}$ (with some fixed basis) such that the set of all $g_i$'s is a generating set of $C_{p}^{r}$. Moreover, let $H=[g_1|\dots |g_n]$ denote the $r\times n$ matrix over the field $\F_p$ (we identify the $g_i$'s with their coordinate vectors, in column-form). Then the code $\mathcal{C}_{S}$ with parity check matrix $H$ is an $[n,n-r]_{p}$-code. And, each $[n,n-r]_{p}$-code can be obtained in this way.
\end{lemma}
\begin{proof}
Since the set of $g_i$'s is a generating set of $C_{p}^{r}$, the matrix $H$ has full rank, i.e. rank $r$. Therefore, $\mathcal{C}_{S}$ is an $(n-r)$-dimensional subspace of $C_{p}^{n}$. This gives the first claim. The second claim follows, since for each $[n,n-r]_{p}$-code there is an  $r\times n$ parity check matrix, which has full rank.
\end{proof}

\begin{lemma}
\label{lem_sad-codes}
Let $d,r \in \N$ and $p$ prime. Then $\s_{A, \le d}(C_p^r)-1$ is equal to the maximum $n$ such that there exists an $[n,n-r]_p$-code of minimal distance at least $d+1$.
\end{lemma}
\begin{proof}
By definition of $\s_{A, \le d}(C_p^r)$ there exists a sequence $S$ over $C_p^r$ of length $\s_{A, \le d}(C_p^r)-1$ that does not have an  $A$-weighted zero-subsum of length at most $d$. Also note that the elements in $S$ generate $C_p^r$; this follows for example from Lemma \ref{lem_sadd}. Thus, $\mathcal{C}_S$ is an $[n,n-r]_p$-code, and it cannot have a non-zero codeword of weight at most $d$; that is, its minimal distance is at least  $d+1$. Conversely, given an $[n,n-r, d']_p$-code with $d' > d$, we get a sequence $S_{\mathcal{C}}$ over $C_p^r$ of length $n$ whose shortest $A$-weighted zero-subsum has length $d'> d$. Thus,  $\s_{A, \le d}(C_p^r)>n$, completing the argument.
\end{proof}

We end this general section on the link between linear codes and $A$-weighted zero-sum problems by pointing out that the existence of $m$ \emph{disjoint} $A$-weighted zero-subsums of some sequence $S$ corresponds precisely to the existence of $m$ non-zero codewords in $\mathcal{C}_S$  with pairwise disjoint supports.
This fits together with a notion considered in coding theory, especially the case $m=2$. We recall that Cohen and Lempel \cite{CL} called a linear code for which any two non-zero codewords do not have disjoint support an intersecting code. We refer to \cite{PS} for a more detailed discussion of (binary) intersecting codes in the current context.

\section{Some exact values and bounds for $\s_{A, \le d}(C_p^r)$}
\label{sec_exactvalue}

We continue to use the notation that  $G$ denotes  a finite abelian group and $A = \{1, \dots, \exp(G)-1\}$.
In this section we establish various results for  $\s_{A,\le d}(C_p^r)$.
The problem of determining this constant is equivalent to a problem in coding theory, as we saw in Section \ref{sec_codesgeneral}. Our results are mainly based on this link and results in coding theory. For the case that $d= 3$, it will however be advantageous to use a further equivalence to reduce the problem to one in discrete geometry.

We start by discussing some simple extremal cases.  It is clear that for $|G|\neq 1$ we always have  $\s_{A, \le 1}(G) = \infty$. Moreover, by \eqref{eq_davA} we have that $\s_{A, \le d} (C_p^r)= r+1$ for $d \ge r+1$.

In the following lemma  we determine  $\s_{A, \le 2} (C_p^r)$ by a simple direct argument.

\begin{lemma}
\label{lem_sa2_p}
Let $p$ be a prime and let $r \in \N$.
Then $\s_{A, \le 2}(C_p^r)= 1 + \frac{p^{r} - 1}{p - 1}$.
\end{lemma}
\begin{proof}
Let $S$ be a sequence over $C_p^r$. Clearly, $S$ has an $A$-weighted zero-subsum of length $1$ if and only if $S$ contains $0$. So, suppose $S$ contains only non-zero elements. Moreover, we observe that $S$ has an $A$-weighted zero-subsum if and only if $S$ contains two elements (including multiplicity) from the same (non-trivial) cyclic subgroup. Thus, the maximal length of a sequence without $A$-weighted zero-subsum of length at most $2$ is equal to the number of non-trivial cyclic subgroups of $C_p^r$, which equals $\frac{p^{r} - 1}{p - 1}$. This implies the claim.
\end{proof}

One could also use the link to coding theory to obtain this result---observe that the  extremal examples correspond to the parity check matrix of $p$-ary Hamming codes---yet in view of the simplicity of a direct argument we preferred to give it.

We proceed to discuss the case that $d=3$. It is well-known and not hard to see that $\s_{A, \le 3}(C_2^r) = 1 + 2^{r-1}$ for each $r\ge 1$ (see \cite[Section 7]{FS} for a more detailed discussion of this and related problems).  We thus restrict to considering $p > 2$. We begin by a further equivalent description of  $\s_{A, \le 3}(C_p^r)$. Recall that a cap set is a subset of an affine or projective space that does not contain three co-linear points. It is well-known that linear codes of minimal distance (at least) four and cap sets are related;
see, e.g., \cite[Section 4]{bierbrauer-edel} or \cite[Section 27.2]{hirschfeld-thas}.
We summarize this relation in a form convenient for our applications in the following lemma. We include the already established relation of our problem on sequences to a problem on codes.

\begin{lemma}
\label{lemma_cap}
Let $p$ be an odd prime, and let $r \ge 3$ and $n \ge 4$ be integers. Let $g_1, \dots, g_n \in C_p^r\setminus \{0\}$ and assume the $g_i$'s generate $C_p^r$.
The following statements are equivalent.
\begin{enumerate}
\item The sequence $g_1 \dots g_n$ has no $A$-weighted zero-subsum of length at most $3$.
\item The $[n,n-r]_p$-code with parity check matrix $[g_1\mid \dots \mid g_n]$ has minimal distance at least $4$.
\item The set of points represented by the $g_i$'s in the projective space of dimension $r-1$ over the field with $p$ elements is a cap set of size $n$.
\end{enumerate}
In particular, the following integers are equal.
\begin{itemize}
\item  $\s_{A, \le 3}(C_p^r) - 1$.
\item The maximal $n$ such that there exists an $[n,n-r]_p$-code of minimal distance at least four.
\item The maximal cardinality of a cap set in the projective space of dimension $r-1$ over $\F_p$.
\end{itemize}
\end{lemma}
\begin{proof}
The equivalence of the first two is merely Lemma \ref{lem_sad-codes}. For the equivalence with the third  it suffices to note that a relation of the form $ag_i + bg_j + cg_k= 0$ with non-zero (modulo $p$) coefficients $a,b,c$  means that the points represented by $g_i,g_j,g_k$ are co-linear; in addition note that  a relation of the form $ag_i + bg_j= 0$ with non-zero (modulo $p$) coefficients $a,b$ would mean that $g_i$ and $g_j$ represent the same point in the projective space, which is excluded by insisting that the size of the cap set is $n$.
\end{proof}

Having this equivalence at hand, there is wealth of results on $\s_{A, \le 3}(C_p^r)$ available. We refer to
\cite{bierbrauer-edel} for a recent survey article on the problem of determining large cap sets in projective spaces. However, only for small dimensions an answer for all $p$ is known. We summarize these results in the current notation.

\begin{theorem}
\label{thm_expl_lowd}
Let $p$ be an odd prime. Then
\begin{enumerate}
\item $\s_{A, \le 3}(C_p)= 2$.
\item $\s_{A, \le 3}(C_p^2)= 3$.
\item $\s_{A, \le 3}(C_p^3)= 2 + p$.
\item $\s_{A, \le 3}(C_p^4)= 2 + p^2$.
\end{enumerate}
\end{theorem}

For larger values of $r$ exact values are only known for small $p$; some of these (for $p \le 7$) are implicitly recalled in the subsequent results. We refer again to \cite{bierbrauer-edel} for more complete information.
Moreover, we point out that data on this problem can be retrieved from databases for the parameters of codes; we mention specifically  \url{www.codetables.de} by Grassl \cite{grassl} and MinT (see \url{http://mint.sbg.ac.at}) by Sch\"urer and Schmid \cite{mint}.

Indeed, recall from Lemma \ref{lem_sad-codes} that to determine $\s_{A, \le d}(C_p^r)-1$ is equivalent to determining the largest $n$ such that an $[n,n-r]_p$-code with minimal distance greater than $d$ exists. Or put differently, determining $\s_{A,\le d}(C_p^r)$ is equivalent to determining the smallest $n$ such that  $d$ is the largest minimum distance of an $[n,n-r]_p$-code.

This type of information can be obtained from data on parameters of codes, though the particular piece of information we need is rather not highlighed in such data.
 
Yet, note that, for some fixed $p$, having a table, such as given in \url{www.codetables.de}, with on one axis the length $n$ of a code, on the other axis the dimension  $k$ of a code, and with entries the maximal minimal distance of an $[n,k]_p$-code (or bounds for it), to find information on $\mathsf{s}_{A,\le d}$ for $C_p^r$ we merely have to look along a suitable diagonal of the table, namely the one with $n-k=r$. 

Moreover, note that using the advanced user interface to MinT one can in fact specify as one of the search-parameters the co-dimension of the code, that is $n-k=r$ (in our notation).

The following results are extracted from the above mentioned data (except for the results for large $d$ that are merely \eqref{eq_davA}). We do not include results for groups of rank one and two as these were already discussed in Lemma \ref{lem_sa2_p}.
However, we do include the results for $d=2$ and $d=3$, already mentioned, to compare the size of the constants.

\begin{theorem}
\label{thm_expl_3} \
\begin{enumerate}
\item $\s_{A, \le 2}(C_3^3) = 14$, $\s_{A, \le 3}(C_3^3) = 5$, and $\s_{A, \le d}(C_3^3)= 4$ for $d \ge 4$.
\item $\s_{A, \le 2}(C_3^4) = 41$, $\s_{A, \le 3}(C_3^4) = 11$, $\s_{A, \le 4}(C_3^4) = 6$,  and $\s_{A, \le d}(C_3^4)= 5$ for $d \ge 5$.
\item $\s_{A, \le 2}(C_3^5) = 122$, $\s_{A, \le 3}(C_3^5) = 21$, $\s_{A, \le 4}(C_3^5) = 12$, $\s_{A, \le 5}(C_3^5) = 7$,  and $\s_{A, \le d}(C_3^5)= 6$ for $d \ge 6$.
\item $\s_{A, \le 2}(C_3^6) = 365$, $\s_{A, \le 3}(C_3^6) = 57$, $\s_{A, \le 4}(C_3^6) = 15$, $\s_{A, \le 5}(C_3^6) = 13$,  $\s_{A, \le d}(C_3^6)= 8$; and $\s_{A, \le d}(C_3^6)= 8$ for $d \ge 7$.
\end{enumerate}
\end{theorem}

\begin{theorem}
\label{thm_expl_5} \
\begin{enumerate}
\item $\s_{A, \le 2}(C_5^3) = 32$, $\s_{A, \le 3}(C_5^3) = 7$, and $\s_{A, \le d}(C_5^3)= 4$ for $d \ge 4$.
\item $\s_{A, \le 2}(C_5^4) = 157$, $\s_{A, \le 3}(C_5^4) = 27$, $\s_{A, \le 4}(C_5^4) = 7$,  and $\s_{A, \le d}(C_5^4)= 5$ for $d \ge 5$.
\item $\s_{A, \le 2}(C_5^5) = 782$, $67\le \s_{A, \le 3}(C_5^5) \le 89$, $\s_{A, \le 4}(C_5^5) = 13$, $\s_{A, \le 5}(C_5^5) = 7$,  and $\s_{A, \le d}(C_5^5)= 6$ for $d \ge 6$.
\end{enumerate}
\end{theorem}

\begin{theorem}
\label{thm_expl_7} \
\begin{enumerate}
\item $\s_{A, \le 2}(C_7^3) =58$, $\s_{A, \le 3}(C_7^3) = 9$, and $\s_{A, \le d}(C_7^3)= 4$ for $d \ge 4$.
\item $\s_{A, \le 2}(C_7^4) =400$, $\s_{A, \le 3}(C_7^4) = 51$, $\s_{A, \le 4}(C_7^4) = 9$,  and $\s_{A, \le d}(C_7^4)= 5$ for $d \ge 5$.
\end{enumerate}
\end{theorem}

In particular, for lager values of $d$ also rather precise information for groups of larger rank could be obtained in this way. However, we do not include this information explicitly.

We also include some information that we can obtain for $p=2$; here, the problem reduces to the classical case. The general point is already established in \cite{CZ}; we merely added the numerical values using known data.  We recall that in this case not only is it known that $\s_{\le 2}(C_2^r) = 2^{r}$ but also that $\s_{\le 3}(C_2^r) = 1 + 2^{r-1}$ (see above). Thus, we can reduce to considering $d\ge 4$ and groups of rank at least $4$.

\begin{theorem}
\label{thm_expl_2} \
\begin{enumerate}
\item $\s_{ \le 4}(C_2^4) = 6$, $\s_{\le d}(C_2^4) = 5$  for $d \ge 5$.
\item $\s_{ \le 4}(C_2^5) = 7$, $\s_{ \le 5}(C_2^5) = 7$, and $\s_{\le d}(C_2^5) = 6$  for $d \ge 6$.
\item $\s_{ \le 4}(C_2^6) = 9$, $\s_{ \le 5}(C_2^6) = 8$,  $\s_{\le 6}(C_2^6) = 8$, and $\s_{\le 6}(C_2^6) = 7$  for $d \ge 7$.
\item $\s_{ \le 4}(C_2^7) = 12$, $\s_{ \le 5}(C_2^7) = 10$, $\s_{ \le 6}(C_2^7) = 9$, $\s_{ \le 7}(C_2^7) = 9$, and $\s_{ \le d}(C_2^7) = 8$ for $d \ge 8$.
\item $\s_{ \le 4}(C_2^8) = 18$, $\s_{ \le 5}(C_2^8) = 13$, $\s_{ \le 6}(C_2^8) = 10$, $\s_{ \le 7}(C_2^8) = 10$,  $\s_{\le 8}(C_2^8) = 10$, and $\s_{ \le d}(C_2^8) = 9$ for $d \ge 9$.
\item $\s_{ \le 4}(C_2^9) = 24$, $\s_{ \le 5}(C_2^9) = 19$, $\s_{ \le 6}(C_2^9) = 12$, $\s_{ \le 7}(C_2^9) = 11$, $\s_{ \le 8}(C_2^9) = 11$, $\s_{ \le 9}(C_2^9) = 11$, and  $\s_{ \le d}(C_2^9) = 10$ for $d \ge 10$.
\item $\s_{ \le 4}(C_2^{10}) = 34$, $\s_{ \le 5}(C_2^{10}) = 25$, $\s_{ \le 6}(C_2^{10}) = 16$, $\s_{ \le 7}(C_2^{10}) = 13$, $\s_{ \le 8}(C_2^{10}) = 12$, $\s_{ \le 9}(C_2^{10}) = 12$, $\s_{ \le 10}(C_2^{10}) = 12$, and  $\s_{ \le d}(C_2^{10}) = 11$ for $d \ge 11$.
\end{enumerate}
\end{theorem}

We end the section by some additional discussion of the case $p=3$ and $d=3$, which is particularly popular. In this case only, we also include a somewhat detailed discussion of asymptotic bounds. We omit such a discussion in the general case; again we refer to \cite{bierbrauer-edel} for additional information.

As mentioned in the introduction Godinho, Lemos, and Marques \cite{GLM} studied some plus-minus weighted zero-sum constants. For groups of exponent $3$, the sets of weights $\{+1, -1\}$ and $\{1, \dots, \exp(G)- 1\}$ are equivalent. Thus, our constant $\s_{A, \le 3} (C_3^r)$ coincides with their $\eta_{A}(C_3^r)$; and as shown there (see Propositions 1 and 2 in \cite{GLM})  their other constants $\s_{A} (C_3^r)$ and $\mathsf{g}_A(C_3^r)$ can be expressed in terms of $\eta_{A}(C_3^r)$, namely $\s_{A} (C_3^r) = \mathsf{g}_A(C_3^r) =  2 \eta_{A}(C_3^r) - 1$.

The results in Theorem \ref{thm_expl_3}  improve on their results, giving the exact value of $\eta_A(C_3^5)$ and $\eta_A(C_3^6)$ in addition. Moreover, in the same way we can obtain the bounds $113 \le \eta_{A}(C_3^7) \le 137$, $249 \le \eta_{A}(C_3^8) \le 387$, $533 \le \eta_{A}(C_3^9) \le 1038$, $1217 \le \eta_{A}(C_3^{10}) \le 2817$, and in fact MinT would contain explict values up to $r=17$, in part stemming from results on cap sets (see the respective entries in MinT for precise references).  

Moreover, using the link established in Lemma \ref{lemma_cap} and using lower bounds on the size of caps in ternary spaces, we get that for sufficiently large $r$ one has
\[\s_{A, \le 3} (C_3^r) \ge 2.217^r\]
and indeed one could take a slightly larger constant (see \cite{bierbrauer-edel} for details). For large $r$ this is considerably better than the bound given in \cite{GLM} (see Propositions 3 and 4 there).
For an upper bound we recall that Bateman and Katz \cite{bateman-katz} recently showed that the maximal size of a cap set in a ternary affine space of dimension $r$ is $O(3^r/ r^{1 + \epsilon})$ for some universal $\epsilon>0$.  This is clearly also an upper bound for the size of a cap set in a ternary projective space of dimension $r-1$  and so
\[\s_{A, \le 3} (C_3^r) = O(3^r/ r^{1 + \epsilon}).\]
The gap between upper and lower bound is significant and even conjecturally it is not at all clear what should be the actual order of magnitude of $\s_{A, \le 3} (C_3^r)$ as $r$ tends to infinity, while the problem, in the equivalent formulation for cap sets, received considerable attention.

\section{All multi-wise fully-weighted Davenport constants for some groups}
\label{sec_allmultiwise}

In the current section we establish the value of all fully-weighted multi-wise constants for certain groups, that is for some $G$ and $A = \{ 1 , \dots , \exp(G) - 1 \}$ we determine $\Dav_{A,m}(G)$ for each $m$. We make use of the results for $\s_{A, \le d}(G)$ established before. The group we consider are on the one hand  elementary $p$-groups of rank at most two, and on the other hand $C_3^3$. We recall that a solution to this  problem for $C_2^r$ for $r\le 5$ is also know (see \cite[Section 7]{FS}).

\begin{theorem}
Let $p$ be a prime number.
\begin{enumerate}
\item $\Dav_{A,m}(C_p) = 2m$, in particular $m_A(C_p)=1$ and $\Dav_{A,0}(C_p)=0$.
\item $\Dav_{A,m}(C_p^2) = 3m$ for $m \le \lceil p/3 \rceil$ and $\Dav_{A,m}(C_p^2) = 2m + \lceil p/3 \rceil$ for $m > \lceil p/3 \rceil$, in particular  $m_A(C_p)= \lceil p/3 \rceil$ and $\Dav_{A,0}(C_p)= \lceil p/3 \rceil$.
\end{enumerate}
\end{theorem}
\begin{proof}
The statement for cyclic groups is a direct consequence of Lemmas \ref{lem_lb_general}, \ref{lem_ub_rec}, and  \ref{lem_sa2_p}.

For $C_p^2$ we note that $\s_{A,\le 2}(C_p^2)=p+2$  by Lemma \ref{lem_sa2_p} and $\s_{A,\le 3}(C_p^2)=3$.
Thus, $\Dav_{A,m}(C_p^2) \le 3m$ for each $m$ by Lemma \ref{lem_ub_rec} and $\Dav_{A,m}(C_p^2) \ge 3m$ for $3m \le p+2$ by Lemma \ref{lem_lb_general}.
Note that $3m \le p+2$ is equivalent to $m \le \lceil p/3 \rceil$.

Now, suppose $m >  \lceil p/3 \rceil$. Write $m = \ell +  \lceil p/3 \rceil$. We need to show $\Dav_{A,m}(C_p^2) = 2m + \lceil p/3 \rceil$, that is $\Dav_{A,m}(C_p^2) = 2 \ell  + 3 \lceil p/3 \rceil$. That this is a lower bound follows from $\Dav_{A,\lceil p/3 \rceil}(C_p^2) =  3 \lceil p/3 \rceil$ and Lemma \ref{lem_lb_+2}.
To show this is an upper bound, let $S$ be a sequence of length $2 \ell  + 3 \lceil p/3 \rceil$. Since $\s_{A, \le 2}(C_p^2)= p + 2$, the sequence $S$ has (at least) $\ell$ disjoint $A$-weighted zero-subsums of length  at most two; let us denote the corresponding sequences by $T_1, \dots, T_{\ell}$. Then the sequence $R=(T_1 \dots T_{\ell})^{-1}S$ has length at least $3 \lceil p/3 \rceil$. Thus, since $\s_{A, \le 2}(C_p^2)= 3$, the sequence $R$ has at least $\lceil p/3 \rceil$ disjoint $A$-weighted zero-subsums. Consequently, we have at least $\ell + \lceil p/3 \rceil=m$ disjoint $A$-weighted zero-subsums of $S$. This shows that $\Dav_{A,m}(C_p^2) \le 2m + \lceil p/3 \rceil$.
\end{proof}

To complement the result for elementary $p$-groups of rank at most two, we  consider the problem for $C_3^3$. For the values of all the classical multi-wise Davenport constants for $C_3^3$ we refer to \cite{BSP}. The result below shows an interesting phenomenon, namely that the difference between $\Dav_{A,m+1}(G) $ and $\Dav_{A,m}(G) $ is not necessarily non-increasing.  

\begin{theorem}
We have $\Dav_{A,1}(C_3^3) = 4$, $\Dav_{A,2}(C_3^3) = 7$, $\Dav_{A,3}(C_3^3) = 9$, $\Dav_{A,4}(C_3^3) = 12$, and $\Dav_{A,m}(C_3^3) = 4 + 2m$ for $m\ge 5$. In particular, $\Dav_{A,0}(C_3^3) = 4$ and $m_{A}(C_3^3) = 5$.
\end{theorem}

\begin{proof}
First, we use results we obtained earlier, to reduce the problem  to showing $\Dav_{A,2}(C_3^3) \ge 7$ and $\Dav_{A,3}(C_3^3) \le 9$.

By \eqref{eq_davA} we get $\Dav_{A,1}(C_3^3)=4$. By Lemma  \ref{lem_lb_general}, with $d=3$ and since $\s_{A,\le 2}(C_3^3) = 14$, see Lemma \ref{lem_sa2_p}, we get the lower bounds for $\Dav_{A,m}(C_3^3)$ for $m \in \{ 3, 4\}$. Then, by Lemma \ref{lem_lb_+2} we also get the lower bounds for each $m \ge 5$.

Since by Theorem \ref{thm_expl_3} we have $\s_{A,\le 3}(C_3^3) = 5$, it follows  by Lemma \ref{lem_ub_rec} that $\Dav_{A,2}(C_3^3) \le 7$.  Moreover, $\s_{A,\le 3}(C_3^3) = 5$ and Lemma \ref{lem_ub_rec} shows that $\Dav_{A,3}(C_3^3) = 9$ implies that $\Dav_{A,4}(C_3^3) \le 12$, and furthermore  $\s_{A,\le 2}(C_3^3) = 14$, see Lemma \ref{lem_sa2_p}, and Lemma \ref{lem_ub_rec} then show $\Dav_{A,m}(C_3^3) \le 4 + 2m$ for $m\ge 5$.

Consequently, it only remains to show that $\Dav_{A,2}(C_3^3) \ge 7$ and $\Dav_{A,3}(C_3^3) \le 9$.

To see that $\Dav_{A,2}(C_3^3)  \ge 7$ we consider the sequence $e_1 e_2 e_3 (e_1+e_2) (e_1+e_3) (e_2+e_3)$. It cannot have an $A$-weighted zero-subsum of length at most $2$, thus if it had two disjoint $A$-weighted zero-subsums they would both be of length $3$. One of the two subsums has to contain at least two of $e_1,e_2,e_3$, say it contains $e_1$ and $e_2$. Then the third element is necessarily $e_1+e_2$. However, the three other elements  $e_3,(e_1+e_3),(e_2+e_3)$ do not have a $A$-weighted zero-subsum. Thus,  $e_1e_2 e_3 (e_1+e_2) (e_1+e_3) (e_2+e_3)$ does not have two disjoint $A$-weighted zero-subsums. Consequently $\Dav_{A,2}(C_3^3)  > 6$.

Now,  we show $\Dav_{A,3}(C_3^3) \le 9$. We reduce this problem to the problem of checking whether three specific sequences have three disjoint $A$-weighted zero-subsums.

First, we recall that any sequence of length $9$ that contains $0$ or an element more than once or an element and its inverse,  has an $A$-weighted subsum of length at most two, and thus, since $\Dav_{A,2}(C_3^3) \le 7$, it has three disjoint $A$-weighted zero-subsums.
Thus, we can restrict to considering squarefree sequences of length $9$ where each cyclic subgroup of $C_3^3$ contains at most one element.

Second, we recall that replacing an element occurring in a sequence by its inverse has no effect on the number of disjoint $A$-weighted zero-subsums.

Thus, we can  restrict to considering subsequences of
\begin{equation*}
\begin{split}
e_1e_2e_3 \ (e_1+e_2)(e_1+e_3)(e_2+e_3) \ (e_1-e_2)(e_1-e_3)(e_2-e_3) \\ (e_1+e_2+e_3) (e_1+e_2-e_3)(e_1-e_2+e_3)(e_1-e_2-e_3);
\end{split}
\end{equation*}
this sequence has length $13$ and contains one non-zero element from each cyclic subgroup.

A subsequence of length $9$ is characterized by the $4$ elements of the above $13$ that it does not contain. But, we do not need to check all sequences resulting from omitting each possible choice of  four elements, since the problem is invariant under isomorphisms of the group.

We argue there are only three cases to consider. Let $g_1, g_2, g_3, g_4$ be four distinct non-zero elements (none the inverse of each other). While below we give a purely algebraic treatment, we remark that we could also consider this as a problem in the two-dimensional ternary projective space; the three cases being four points on a line, three (yet not four) on a line, and no three co-linear points (i.e., a cap-set).    

Case 1: $g_1, g_2, g_3, g_4$ do not generate $C_3^3$. In this case, we can assume that $g_3 = g_1 + g_2$ and $g_4 = g_1 - g_3$. (Certainly, $g_3,g_4$ can be written as a linear combination of $g_1, g_2$ and since we are free to choose signs, this is the only possibility we need to consider.)

Case 2: $g_1, g_2, g_3, g_4$ generate $C_3^3$, but there is some $g_j$, say $g_3$, such that the set $\{g_1,g_2,g_3,g_4\}\setminus \{g_3\}$ does not generate the group.
We observe that $g_1, g_2, g_3$ is a generating set; $g_1, g_2$ are independent by assumption, while $g_3$ is not in the subgroup generated $g_1$ and $g_2$. Moreover, $g_4$ is an element of the group generated by $g_1$ and $g_2$, and as we can ignore signs, we can assume $g_4 = g_1 - g_2$.

Case 3: $g_1,g_2,g_3,g_4$ generate $C_3^3$, and in fact $\{g_1,g_2,g_3,g_4\}\setminus \{g_j\}$ generates the group for each $j$. We note that $g_1,g_2,g_3$ is a generating set and $g_4= a_1g_1 + a_2g_2+ a_3g_3$. Since $\{g_1,g_2,g_3,g_4\}\setminus \{g_j\}$ is a generating set for each $j$ it follows that all $a_i$ are non-zero and therefore, as signs are irrelevant, we can assume $g_4 = - (g_1 + g_2 + g_3)$.

Since isomorphisms preserve $A$-weighted zero-subsums we can choose for the independent elements whatever independent elements we like. In case 1 we choose $g_1 = e_2$ and $g_ 2 = e_3$. Thus after removing the four elements $g_1, g_2, g_3, g_4$ the following sequence remains:
\[e_1(e_1+e_2)(e_1+e_3) (e_1-e_2)(e_1-e_3) (e_1+e_2+e_3) (e_1+e_2-e_3)(e_1-e_2+e_3)(e_1-e_2-e_3).\]
In case 2 and 3 we chose $g_1 = e_1 - e_2 + e_3$ and $g_2 = e_1 + e_ 2 - e_3$ and $g_3 = e_1 - e_2 - e_3$.

Thus after removing the four elements $g_1, g_2, g_3, g_4$ the following sequence remains in case 2:
\[e_1e_2e_3  (e_1+e_2)(e_1+e_3)(e_2+e_3)  (e_1-e_2)(e_1-e_3)  (e_1+e_2+e_3) .\]

And, in case 3:
\[e_1e_2e_3 (e_1+e_2)(e_1+e_3)  (e_1-e_2)(e_1-e_3)(e_2-e_3)  (e_1+e_2+e_3).\]

Now, it remains to check that these three sequences of length $9$ each have $3$ disjoint $A$-weighted zero-subsums. If this is established it follows from the above arguments that in fact each sequence of length $9$ over $C_3^3$ has this property, and thus $\Dav_{A,3}(C_3^3)\le 9$.

We give the decompositions into $A$-weighted subsums of length $3$, where for clarity we write each element in parenthesis:
\[
\begin{split} 
& (e_1)   + (e_1+e_2) + (e_1-e_2), (e_1+e_3) + (e_1+e_2+e_3) + (e_1-e_2+e_3), \\
& (e_1-e_3) + (e_1+e_2-e_3) + (e_1-e_2-e_3) 
\end{split}
\]

\[(e_1)   + (e_2+e_3) - (e_1+e_2+e_3), (e_2) - (e_1-e_2) + (e_1+e_2), ( e_3) + (e_1+e_3) - (e_1-e_3) \]

\[ (e_1)   + (e_3) - (e_1+e_3) ,  (e_2) + (e_1+e_2) - (e_1-e_2), (e_1-e_3) + (e_2-e_3) - (e_1+e_2+e_3)\]

\end{proof}

\section{Some asymptotic results}
\label{sec_asymptotic}

In the current section we complement the earlier results focused on small groups with results focused on elementary $p$-groups of larger rank.  Again, we use the link to coding  theory explained in Section \ref{sec_codesgeneral}.

We use some ad-hoc terminology based on the one introduced in \cite{PS}. A function $f:[0,s]\rightarrow [0,1]$, for $0< s \le 1$, is called $p$-upper-bounding if it is non-increasing, continuous, and each $[n,k,d]_{p}$-code with $d/n$ in $[0,s]$ satisfies
\[
\frac{k}{n}\leq f\left(\frac{d}{n}\right)
\]
And, $f$ is called asymptotically $p$-upper-bounding if $\frac{k}{n}\lesssim  f\left(\frac{d}{n}\right)$
for $[n,k,d]_{p}$-codes with sufficiently large $n$.

Thus, (asymptotically) $p$-upper-bounding functions are the functions in the (asymptotic) upper bounds for the rate $k/n$ of a $p$-linear code as a function of its normalized minimal distance $d/n$.

The following is essentially a reformulation of Lemma \ref{lem_sad-codes}, in a way more suitable for the current applications.

\begin{lemma}
\label{lem_bounding_basic}
Let $f$  be a $p$-upper-bounding function, and let $d, n, r\in \N$ with $2\leq d \leq n-1$ and
\[
\frac{n-r}{n} > f \left(\frac{d+1}{n}\right)
\]
then $\s_{A,\leq d}(C_{p}^{r})\le n$. Moreover, the same holds true for $f$ an asymptotically $p$-upper-bounding function and sufficiently large $n$ such that the inequality holds uniformly in $n$ (that is $\frac{n-r}{n} >  f \left(\frac{d+1}{n}\right)  + \varepsilon$ for some $\varepsilon>0$ independent of $n$).
\end{lemma}
\begin{proof}
Let $S=g_1 \dots g_n$ be an arbitrary sequence over $C_{p}^{r}$. By Lemma \ref{lem_sadd} we know that $\mathsf{s}_{A, \le d} (C_p^r) \ge \mathsf{s}_{A, \le d} (C_p^s)$ for $s \le r$, and we thus can assume that the $g_i$ generate $C_p^r$. 

We choose some basis of $C_{p}^{r}$ and $C_{p}^{n}$. We apply Lemma \ref{lem_codematrix} with the sequence $S$ to get an $[n,n-r]_{p}$-code $\mathcal{C}_{S}\subset C_{p}^{n}$. Let $e$ be the minimal distance of $\mathcal{C}_{S}$, i.e $\mathcal{C}_{S}$ is an $[n,n-r,e]_{p}$-code. But by assumption since $f$ is upper-bounding and 
\[
\frac{n-r}{n} > f \left( \frac{d+1}{n} \right)
\]
an $[n,n-r,d+1]_{p}$-code cannot exist. This implies that $e< d+1$, or equivalently $d \geq e$. We conclude by applying Lemma \ref{lem_codebasic}, which shows that $S$ possesses an $A$-weighted zero-sum subsequence of length $e$.

The additional claim, for asymptotically upper-bounding functions, is immediate in view of the just given argument and the fact that our condition just  negates $\frac{n-r}{n}\lesssim  f\left(\frac{d+1}{n}\right)$
\end{proof}

We recall the following well-known fact that we need in the subsequent arguments (see, e.g., \cite[Appendix B.3]{MWS}).

\begin{lemma}
\label{lemma_subspaces}
Let $k, n \in \N$ with $n\geq k$. In an $n$-dimensional vector space over a field with $p$ elements, the number of $k$-dimensional subspaces is equal to the $p$-ary binomial coefficient defined as
\[
{n \brack  k}_{p} = \frac{(p^{n}-1)\ldots (p^{n-k+1}-1)}{(p^{k}-1)\ldots (p-1)}
\]
Moreover, the number of $k$-dimensional subspaces containing a fixed $j$-dimensional subspace, $k\geq j$, is equal to
\[
{n - j \brack  k - j}_{p}.
\]
\end{lemma}

Now, we state one of the main results of this section, a lower bound for the fully-weighted multi-wise Davenport constants for elementary $p$-groups of large rank.

\begin{proposition}
\label{propgeneral}
Let $m \in \N$ and let $p$ be a prime number. Then, for sufficiently large $r$, with $A = \{1, \dots, p-1\}$,
\[
\Dav_{A,m}(C_{p}^{r})\geq \log p \frac{m}{\log(1 + m(p-1))}r.
\]
\end{proposition}

\begin{proof}
For $m=1$ we know by \eqref{eq_davA} that $\Dav_{A,m}(C_{p}^{r})= r+1$. Since $\log p /(\log(1+1(p-1))) \le 1$, the claim follows. Now, we fix a positive integer $m>1$.

Since $m\log p /(\log(1 + m(p-1)))> 1$ and $r$ is sufficiently large,  there is an integer $n$ such that
\[
r+m \leq n < \log p \frac{m}{\log(1 + m(p-1))}r
\]
Recall that by Lemma \ref{lem_sad-codes} we can associate to each sequence $S$ of length $n$ over $C_{p}^{r}$ whose elements generate $C_{p}^{r}$ an $[n,n-r]_p$-code, and indeed we can obtain every such code in this way.

We observe that the condition that  $S$ has $m$ disjoint $A$-weighted zero-sum  subsequences translates to the condition that the associated code contains $m$ non-zero codewords $c_1, \dots, c_m$ such that pairwise intersections of their supports are empty. We call such a code $m$-inadmissible, otherwise it will be called $m$-admissible.

We first produce an upper bound on the total number of $[n,n-r]_p$-codes that are $m$-inadmissible.

By definition any $m$-inadmissible code contains $c_1,\dots,c_m$ with the above mentioned property. These $c_{i}$'s generate an $m$-dimensional vector space; to see this just note that the non-zero coordinates of  each $c_i$ are unique to that element and thus the $c_{m}$'s are certainly independent.

Let $\mathcal{V}$ denote the set of all subsets $\{d_{1},\dots,d_{m}\}\subset C_{p}^{n}\setminus \{0\}$ such that the intersection of the support of $d_{u}$ and $d_{v}$ is empty for all distinct $u,v \in \{1,\dots, m\}$; thus, in particular, all the $d_{i}$'s are distinct.

We note that a code $\mathcal{C}$ is $m$-inadmissible if and only if $V\subset \mathcal{C}$ for some $V\in \mathcal{V}$ (this $V$ is not necessarily unique). Moreover, Lemma \ref{lemma_subspaces} implies that for each $V\in \mathcal{V}$ there are
${n-m \brack n-r-m}_{p}$ codes containing $V$; to see this note that if $V\subset \mathcal{C}$ then $\mathcal{C}$ also contains the vector space generate by $V$, which is $m$-dimensional, and apply Lemma \ref{lemma_subspaces}. It thus  follows that the total number of $m$-inadmissible codes cannot exceed
\[
|\mathcal{V}|{n-m \brack n-r-m}_{p}.
\]
We give a simple estimate for $|\mathcal{V}|$.
An element $\{d_{1},\dots ,d_{m}\}$ of $\mathcal{V}$ can be described by specifying for each $l \in \{1,\dots, n\}$  which of the supports of the $d_i$'s (if any) contains $l$ and  which (non-zero) value the respective coordinate has. Thus, for each $l$ there are $1 + m (p-1)$ possibilities, and consequently $(1 + m (p-1))^n$ is an upper bound for $|\mathcal{V}|$.

We therefore infer that the total number of $m$-inadmissible $[n,n-r]_p$-codes is bounded above by
\[
(1 + m (p-1))^n {n-m \brack  n - r - m  }_p.
\]
Again, by Lemma \ref{lemma_subspaces}, it follows that the ratio of total number of $m$-inadmissible $[n,n-r]_p$-codes divided by total number of $[n,n-r]_p$-codes is bounded above by
\[\begin{aligned}
\frac{(1 + m (p-1))^n {n-m \brack  n - r - m  }_{p}}{{n \brack  n - r  }_{p}} & = (1 + m (p-1))^n \prod_{k=n-m+1}^{n}\frac{p^{k-r}-1}{p^{k}-1} \\
              & \leq  (1 + m (p-1))^n \prod_{k=n-m+1}^{n}\frac{p^{k-r}}{p^{k}}\\
              & =   (1 + m (p-1))^n p^{-rm}\\
              & =  p^{n\log_{p}(1 + m (p-1))-rm}
\end{aligned}\]
Thus, it follow that as soon as $(n\log_{p}(1 + m (p-1))-rm)$ is negative, that is
\[
\frac{n}{r}< \frac{m}{\log_{p}(1 + m (p-1))}
\]
the existence of at least one admissible code is guaranteed.
From this we deduce
\[
\mathsf{D}_{A,m}(C_{p}^{r})\geq \log p \frac{m}{\log(1 + m (p-1))}r
\] for sufficiently large $r$.
\end{proof}

The following result combines Lemmas \ref{lem_ub_rec} and \ref{lem_bounding_basic}.
\begin{lemma}
\label{lem_induct}
Let $m,r \in \N$ and let $p$ be a prime number, and let $A= \{1, \cdots , p-1\}$. Furthermore, let $f$ be an asymptotic upper-bounding function.
\begin{enumerate}
\item If  $\Dav_{A,m}(C_{p}^{r})\leq br$ for each sufficiently large $r$ and $c$ is a solution to the inequality
\[
\frac{b+c-1}{b+c} > f \left(\frac{c}{b+c}\right),
\]
then for each sufficiently large integer $r$, we have $\Dav_{A,m+1}(C_{p}^{r})\leq (b+c)r$.
\item If $\Dav_{A,m}(C_{p}^{r})\lesssim br$ and $c$ is a solution to the inequality
\[
\frac{b+c-1}{b+c}\ge f \left(\frac{c}{b+c}\right),
\]
then we have $\Dav_{A,m+1}(C_{p}^{r}) \lesssim (b+c)r$.
\end{enumerate}
\end{lemma}
\begin{proof}
We start by proving 1. Given the assumptions, we have
\[
\frac{(b+c)r-r}{(b+c)r}=\frac{b+c-1}{b+c}>f\left(\frac{c}{b+c}\right)=f\left(\frac{cr}{(b+c)r}\right)\geq f\left(\frac{cr+1}{(b+c)r}\right),\]
where we used that  $f$ is  non-increasing.  
Clearly, the inequality holds uniformly in $r$.
Replacing $\lfloor (b+c)r\rfloor$ by $n$ and $\lfloor cr \rfloor$ by $d$, we obtain (for $r$ sufficiently large and by the continuity of $f$) that
\[
\frac{n-r}{n} > f \left( \frac{d+1}{n} \right).
\]
And, the inequality still holds uniformly. By Lemma \ref{lem_bounding_basic} we have
\[
\s_{A, \le \lfloor cr \rfloor}(C_{p}^{r})\le n = \lfloor (b+c)r\rfloor \leq (b+c)r
\]
and  then by Lemma \ref{lem_ub_rec}
\[
\Dav_{A,m+1}(C_{p}^{r})\leq \min_{i\in \mathbb{N}}\max\{\Dav_{A,m}(C_{p}^{r})+i, \s_{A, \le i}(C_{p}^{r})\}.
\]
Finally,
\[\begin{aligned}
\Dav_{A,m+1}(C_{p}^{r}) & \le  \min_{i\in \mathbb{N}}\max\{ \Dav_{A,m}(C_{p}^{r})+i, \s_{A, \le i}(C_{p}^{r})\}\\
												& \le  \max\{br+\lfloor cr \rfloor, \s_{A, \le \lfloor cr \rfloor}(C_{p}^{r})\}\\
                        & \le  \max\{(b+c)r,(b+c)r\}\\
                        &  = (b+c)r
\end{aligned}\]
showing the claim in 1.

Now,  let $\varepsilon > 0$ and assume the conditions in 2. are fulfilled. Since $\frac{b+c-1}{p+c}\ge f(\frac{c}{b+c})$, we get that (the left hand-side increases while the right-hand side does not increase)
\[
\frac{b+c +\varepsilon/2 - 1}{b+c+\varepsilon/2} > \frac{b+c  - 1}{b+c} \ge f \left(\frac{c}{b+c} \right) \ge f \left( \frac{c+\varepsilon/2}{b+c+\varepsilon/2} \right).
\]
As above we thus get, for sufficiently larger $r$, that
\[
\s_{A, \le \lfloor (c +\varepsilon/2)  r \rfloor}(C_{p}^{r})\le n=\lfloor (b+(c +\varepsilon/2))r\rfloor \leq (b+(c +\varepsilon/2))r.
\]
A sequence $S$ of length at least $(b + c + \varepsilon)r$ thus contains a subsequence $T$ of length at most $(c + \varepsilon/2) r$ that has $0$ as an   $A$-weighted  sum. Since the length of $T^{-1}S$ is at least $(b + \varepsilon/2)r$ and since we assumed  $\Dav_{A,m}(C_{p}^{r}) \lesssim b r$ it follows that (if $r$ is sufficiently large) the sequence  $T^{-1}S$ admits $m$ disjoint $A$-weighted zero-subsums. Thus, $S$ admits $m+1$ disjoint $A$-weighted zero-subsums, showing that
$\Dav_{A,m+1}(C_{p}^{r}) \le (b + c + \varepsilon)r $ for all sufficiently large $r$. Consequently, $\Dav_{A,m+1}(C_{p}^{r}) \lesssim (b + c)r$.
\end{proof}

We use the just established lemma in two ways. First, we give bounds for $\Dav_{A,m}(C_{p}^{r})$ for small $p$ and $m$ yet large $r$. Then, in Theorem \ref{upper_bound_th} we investigate the asymptotic behavior of $\Dav_{A,m}(C_{p}^{r})$ for large $m$ and $r$; recall that we studied the problem for fixed $r$ and large $m$ in Theorem \ref{thm_arithprogr}.

Of course, to apply Lemma \ref{lem_induct} we need some asymptotic $p$-upper bounding function. We recall some asymptotic bounds on the parameters of codes that we use (see for example \cite[Section 2.10]{huffman-pless}). For a prime $p$ and $0 < x \le (p-1)/p$, let
\[h_p (x) =  -x \log_p (x/(p-1)) - (1 - x) \log_p (1-x) \]
and $h_p(x)=0$, denote the $p$-ary entropy function.
The following functions are $p$-upper bounding functions on $[0, (p-1)/p]$:
\begin{enumerate}
\item \[1 - h_p \left( \frac{x}{2} \right) \] by the asymptotic Hamming bound.
\item \[1 -  h_p \left( \frac{p-1}{p} -  \sqrt{\frac{p-1}{p} \left( \frac{p-1}{p} - x \right)} \right) \] by the asymptotic Elias bound.
\item \[h_p \left( \frac{ p - 1 - (p - 2) x - 2 \sqrt{(p - 1) x (1-x) }}{p} \right) \] by the first MRRW bound.
\end{enumerate}

We now formulate the result for small $p$ and $m$; as the proof shows, we could obtain similar results for further values.  We recall from \eqref{eq_davA} that $\Dav_{A,1}(C_{p}^{r})=r+1$, which is why we do not include this case. Moreover, the case $p=2$ was considered in \cite{PS} and we do not repeat the result.

\begin{theorem}
\label{thm_asymp_small}
For each sufficiently large integer $r$ we have:
\begin{enumerate}
\item 
\[
\begin{aligned}
 1.365\,r\leq & \Dav_{A,2}(C_{3}^{r})& \leq 1.549\,r\\
 1.693\,r\leq & \Dav_{A,3}(C_{3}^{r})& \leq 2.085\,r\\
 2\,r\leq & \Dav_{A,4}(C_{3}^{r})& \leq 2.610\,r\\
 2.290\,r\leq & \Dav_{A,5}(C_{3}^{r})& \leq 3.112\,r
\end{aligned}
\]
 \item \[
\begin{aligned}
 1.464\,r\leq & \Dav_{A,2}(C_{5}^{r})& \leq 1.699\,r\\
 1.882\,r\leq & \Dav_{A,3}(C_{5}^{r})& \leq 2.397\,r\\
 1.272\,r\leq & \Dav_{A,4}(C_{5}^{r})& \leq 3.065\,r\\
 2.643\,r\leq & \Dav_{A,5}(C_{5}^{r})& \leq 3.707\,r
\end{aligned}
\]
\item \[
\begin{aligned}
 1.517\,r\leq & \Dav_{A,2}(C_{7}^{r})& \leq 1.779\,r\\
 1.982\,r\leq & \Dav_{A,3}(C_{7}^{r})& \leq 2.563\,r\\
 2.418\,r\leq & \Dav_{A,4}(C_{7}^{r})& \leq 3.311\,r\\
 2.833\,r\leq & \Dav_{A,5}(C_{7}^{r})& \leq 4.032\,r
\end{aligned}
\]
\end{enumerate}
\end{theorem}

\begin{proof} The lower bounds are merely derived from Proposition \ref{propgeneral}, rounding \emph{down} the exact value. For the upper bounds we apply repeatedly Lemma \ref{lem_induct}. Since $\Dav_{A,1}(C_{p}^{r})=r+1$, for any fixed $b>1$, we have $\Dav_{A,1}(C_{p}^{r}) < br $; we take $b_1=1.001$ as starting value.  Numerically, we find a solution $c_1^f$ to the inequality $\frac{b_1+ c - 1}{b_1 + c}>f(\frac{c}{b_1 + c})$ for $f$ one of the upper-bounding functions mentioned above; in practice we find a solution of the inequality and round it \emph{up}. We then know $\Dav_{A,2}(C_{p}^{r}) \le (b_1 + c_{1,f})r$.  For $f$ the upper-bounding function that yields the smallest $c_{1,f}$, we set $b_2 = b_1 + c_{1,f}$. (In fact, in this case this is always the first MRRW bound but later it can also be the asymptotic Elias bound.) We have $\Dav_{A,2}(C_{p}^{r}) \le b_2 r$ and this is the bound we give in the result. We then proceed in the same way to get a bound for  $\Dav_{A,3}(C_{p}^{r})$, and so on.
\end{proof}

We continue by investigating the behavior of $\Dav_{A,m}(C_{p}^{r})$ for large $m$ and $r$.

\begin{theorem}
\label{upper_bound_th}
Let $p$ be a prime number and $A=\{1, \cdots , p-1\}$. When $m$ tends to infinity, we have
\[
\limsup_{r \rightarrow +\infty} \frac{\Dav_{A,m}(C_{p}^{\,r})}{r} \lesssim 2\log p\frac{m}{\log m}.
\]
\end{theorem}
\begin{proof}
We apply Lemma \ref{lem_induct} with the function
\[
1 - h_{p}\left( \frac{x}{2} \right),
\]
which is  an asymptotic $p$-upper bounding function by the asymptotic Hamming bound (see above).

Recursively, we define a sequence  $(v_{m})_{m\in \mathbb{N}}$. We set  $v_1=1$ and we define $v_{m+1}$ as the smallest positive real such that (where $V_{m}=v_{1}+\dots +v_{m}$)
\begin{equation}
\label{def_vj}
\frac{1}{V_{m}+v_{m+1}}=h_{p}\left(\frac{v_{m+1}}{2(V_{m}+v_{m+1})}\right).
\end{equation}
We note that this is well-defined and that $v_m \le 2 (p-1) / p$ for each $m$; recall  that $h_p$ is convex and attains its maximum of $1$ at $(p-1)/p$.

Note that by  Lemma \ref{lem_induct}  $\Dav_{A,m}(C_{p}^{r}) \lesssim V_mr$.  Observe that from this and Proposition \ref{propgeneral} it follows that $V_m \gg m/\log m$ for $m \to \infty$.

We proceed to investigate the thus defined quantities. Multiplying \eqref{def_vj} by  $(2 \log p) (V_m + v_{m+1})/v_{m+1}$, gives
\[
\begin{aligned}
& \frac{2\log p}{v_{m+1}}  = \\
& -\log\left(\frac{v_{m+1}}{2(V_{m}+v_{m+1})(p-1)}\right)  + \left(  1-\frac{  2(V_{m}+v_{m+1}) }{ v_{m+1}}   \right ) \log\left(1-\frac{v_{m+1}}{2(V_{m}+v_{m+1})}\right) = \\
& - \log\left(\frac{v_{m+1}}{2(V_{m}+v_{m+1})(p-1)}\right)  + O(1)
\end{aligned}
\]
where we used that $( 1 - y ) ( \log (1 - y^{-1}) )= O(1)$ for $y \to \infty$ and that $v_m$ is bounded while $V_m \to \infty $ as $m \to \infty$.
Consequently, for the second equality using again that $v_m$ is bounded while $V_m \to \infty $ as $m \to \infty$,
\[
\frac{2\log p}{v_{m+1}}= \log(V_{m} + v_{m+1})-\log(v_{m+1}) + O(1)  = \log(V_{m}) - \log(v_{m+1}) + O(1).
\]
If follows that, as $m$ tends to infinity,
\begin{equation}
\label{eq_asymp_vm+1}
v_{m+1} \sim \frac{2\log p}{\log V_{m}}.
\end{equation}

Using again $V_m \gg m/\log m$, it follows that 
\[
V_{m+1}-V_{m}=v_{m+1} \lesssim \frac{2\log p}{\log m}
\]
and therefore, summing all these estimates yields
\[
V_{m} \lesssim 2\log p \sum_{k=1}^{m-1}\frac{1}{\log k} \sim 2\log p\frac{m}{\log m}
\]
establishing the claimed upper bound.
\end{proof}

Combining the lower and the upper bound for $\Dav_{A,m}(C_p^r)$ from Proposition \ref{propgeneral} and the theorem above we get that
\[ \log p\frac{m}{\log m}  \lesssim  \limsup_{r \rightarrow +\infty} \frac{\Dav_{A,m}(C_{p}^{\,r})}{r} \lesssim 2\log p\frac{m}{\log m}.\]
The lower bound seems more likely to give the correct growth. For some discussion of this in the case of $p=2$, we refer to \cite{PS}.

\section*{Acknowledgment} The authors would like to thank the referees for numerous detailed remarks and suggestions.

\end{document}